\documentclass{amsart}

\usepackage[letterpaper,top=2cm,bottom=2cm,left=3cm,right=3cm,marginparwidth=1.75cm]{geometry}

\usepackage[utf8]{inputenc}
\usepackage{amsmath,amsthm, amssymb}
\usepackage{tikz,color}
\usepackage{hyperref}
\usepackage{todonotes}
\usetikzlibrary{calc,math}
\usetikzlibrary{intersections}
\usepackage{graphicx,import}
\usepackage{cleveref}

\newtheorem{theorem}{Theorem}[section]
\newtheorem{corollary}[theorem]{Corollary}
\newtheorem{proposition}[theorem]{Proposition}
\newtheorem{lemma}[theorem]{Lemma}

\theoremstyle{definition}
\newtheorem{definition}[theorem]{Definition}
\newtheorem{example}[theorem]{Example}

\newtheorem{remark}[theorem]{Remark}

\definecolor{blue}{rgb}{0, 0.445, 0.695}
\definecolor{bluishgreen}{rgb}{0, 0.626, 0.456}
\definecolor{red}{rgb}{0.896, 0.395, 0}
\definecolor{purple}{rgb}{0.783, 0.464, 0.640}
\definecolor{skyblue}{rgb}{0.359, 0.752, 0.973}
\definecolor{orange}{rgb}{0.999, 0.706, 0.0}
\definecolor{yellow}{rgb}{0.937, 0.890, 0.258}

\definecolor{olive}{RGB}{116,141,19}
\definecolor{green}{RGB}{108,208,48}
\definecolor{teal}{RGB}{47,77,62}
\definecolor{turquoise}{RGB}{86,235,211}
\definecolor{lightblue}{RGB}{150,178,153}
\definecolor{blue2}{RGB}{25,50,191}
\definecolor{indigo}{RGB}{142,128,251}
\definecolor{indigo2}{RGB}{114,32,246}
\definecolor{lightpurple}{RGB}{243,197,250}
\definecolor{purple2}{RGB}{105,66,131}
\definecolor{magenta}{RGB}{206,43,188}
\definecolor{brown}{RGB}{110,57,13}

\DeclareMathOperator{\alt}{\operatorname{alt}}
\DeclareMathOperator{\conv}{\operatorname{conv}}

\DeclareMathOperator{\area}{\operatorname{area}}
\newcommand{\Tam}[1]{\operatorname{Tam}({#1})}
\newcommand{\altTam}[2]{\operatorname{Tam}_{#1}(#2)}
\newcommand{\TamComplex}[1]{\mathcal{TC}({#1})}
\newcommand{\altTamComplex}[2]{\mathcal{TC}_{#1}(#2)}
\newcommand{\genPerm}[1]{\mathcal{P}_{#1}}
\newcommand{\MixedSubdivision}[1]{\mathcal{M}({#1})}
\newcommand{\Asso}[1]{\operatorname{Asso}_{#1}}
\newcommand{\altAsso}[2]{\operatorname{Asso}_{#1,#2}}
\newcommand{\drot}[1]{\lessdot_{#1}}

\newcommand{\rightflushing}[2]{\varphi_{#1,#2}}
\newcommand{\uu}{\mathbf{u}}
\newcommand{\subpolytope}[1]{S_{#1}} 
\newcommand{\arrangement}{\mathcal{H}^h}

\newcommand{\defn}[1]{{\color{green!50!black}\emph{#1}}}

\usepackage{ifthen}
\usepackage{tikz,pgf,pgfmath}
\usetikzlibrary{decorations.pathmorphing} 
	\usetikzlibrary{shapes,calc,decorations.pathmorphing}
	\pgfdeclarelayer{background}
	\pgfdeclarelayer{middle}
	\pgfsetlayers{background,middle,main}
	\tikzstyle{edge}=[line width=.75pt]
	\tikzstyle{fnode}=[fill=black,draw=black,circle,scale=\s]
	\tikzstyle{pathnode}=[inner sep=.9pt]

	\newcommand{\red}{red!50!white}
	\newcommand{\blue}{blue!50!white}

	\newcommand{\green}{green!50!gray}
	\newcommand{\gray}{white!50!gray}
	\tikzstyle{facet}=[fill=\green,fill opacity=0.6]
	\tikzstyle{facetR}=[fill=\red,fill opacity=0.8]
	\tikzstyle{facetL}=[fill=\blue,fill opacity=0.8]
	\tikzstyle{pathnode}=[inner sep=.9pt]

        \newcounter{pick}
        \newcounter{valley}
        \newcounter{valleyleft}
        \newcounter{height}
        \newcounter{previousheight}
        \newcounter{totalheight}
        \newcounter{parentheight}
        \newcounter{width}
        \newcounter{widthzero}
        \newcounter{relativedistance}
        \newcounter{toremove}
        \newcounter{firstremove}

\newcommand{\nuPath}[7]{
    \begin{tikzpicture}[scale=#3,
                    pathedge/.style={line width=1.3pt}]
		\setcounter{pick}{0}
		\setcounter{valley}{0}
		\setcounter{height}{0}
		\setcounter{previousheight}{0}
        \setcounter{totalheight}{-1}
        
            \foreach \step in {#1}{
                \addtocounter{totalheight}{1}
            }
            \draw[\gray, dashed] (0,0) -- (0,\thetotalheight);
            \foreach \step in {#1}{
                \addtocounter{valley}{\step}
                \draw[\gray, dashed] (0,\theheight) -- (\thevalley,\theheight); 
                    \ifthenelse{\thepick=\thevalley}{}{
                        \pgfmathparse{\thepick+1}
                        \let\respick\pgfmathresult
                        \foreach \x in {\respick,...,\thevalley}{
                            \draw[\gray, dashed] (\x,\theheight) -- (\x,\thetotalheight); 
                        }
                    }
                \addtocounter{pick}{\step}
                \addtocounter{height}{1}
            }

        \setcounter{pick}{0}
		\setcounter{valley}{0}
		\setcounter{height}{0}
        \foreach \step in {#2}{
            \addtocounter{valley}{\step}
            \draw[pathedge] (\thepick,\thepreviousheight) -- (\thepick,\theheight) -- (\thevalley,\theheight);
            \addtocounter{pick}{\step}
            \setcounter{previousheight}{\theheight}
            \addtocounter{height}{1}
        }

 		\foreach \a/\b/\c/\d in {#4}{
			\ifthenelse{\equal{\c}{}}{}{
				\draw[\c,line width=.5pt](\a,\b) circle(.25);
			}
			\fill[\d,line width=.5pt](\a,\b) circle(.15);
 		}
		\foreach \a/\b/\c/\d/\e in {#5}{
			\draw[\d](\a,\b) node[scale=.9,anchor=\e]{\c};
		}
        #7
    \end{tikzpicture}
}

\newcommand{\nuTree}[7]{
    \begin{tikzpicture}[scale=#3,
            fillnode/.style={fill=black,draw=black,circle,scale=1},
            circlenode/.style={draw=black,circle,scale=1},
            treeedge/.style={line width=0.9pt,black},
    ]
		\setcounter{pick}{0}
		\setcounter{valley}{0}
		\setcounter{height}{0}
		\setcounter{previousheight}{0}
        \setcounter{totalheight}{-1}

        #7
        
            \foreach \step in {#1}{
                \addtocounter{totalheight}{1}
            }
            \draw[\gray, dashed] (0,0) -- (0,\thetotalheight);
            \foreach \step in {#1}{
                \addtocounter{valley}{\step}
                \draw[\gray, dashed] (0,\theheight) -- (\thevalley,\theheight); 
                    \ifthenelse{\thepick=\thevalley}{}{
                        \pgfmathparse{\thepick+1}
                        \let\respick\pgfmathresult
                        \foreach \x in {\respick,...,\thevalley}{
                            \draw[\gray, dashed] (\x,\theheight) -- (\x,\thetotalheight); 
                        }
                    }
                \addtocounter{pick}{\step}
                \addtocounter{height}{1}
            }

		\setcounter{height}{0}
  		\setcounter{width}{0}

        \edef\mynu{{#1}}
        \edef\mymu{{#2}}

            \pgfmathsetmacro{\stepnu}{\mynu[0]}
            \pgfmathsetmacro{\stepmu}{\mymu[0]}
            \addtocounter{width}{\stepnu}

            \foreach \x in {0,...,\stepmu}{
                \node[fillnode,scale=#3] at (\thewidth-\x,\theheight) {};
            }
            
            \draw[treeedge] (\thewidth-\stepmu,\theheight) -- (\thewidth,\theheight);

            \ifthenelse{0=\thetotalheight}{}{
            \setcounter{parentheight}{\theheight}
            \setcounter{relativedistance}{0}
            \foreach \k in {0,...,\thetotalheight}{
                \addtocounter{parentheight}{1}
                
                \pgfmathsetmacro{\substepnuabove}{\mynu[\k+1]}
                \pgfmathsetmacro{\substepmuabove}{\mymu[\k+1]}
                \addtocounter{relativedistance}{\substepnuabove-\substepmuabove}
                
                \ifthenelse{\therelativedistance>0}{}{\breakforeach} 
            }            
            \draw[treeedge] (\thewidth-\stepmu,\theheight) -- (\thewidth-\stepmu,\theparentheight);
            }

            \addtocounter{height}{1}

        \foreach \j in {1,...,\thetotalheight}{

            \pgfmathsetmacro{\stepnu}{\mynu[\j]}
            \pgfmathsetmacro{\stepmu}{\mymu[\j]}
            \addtocounter{width}{\stepnu}

            \foreach \x in {0,...,\stepmu}{
                \setcounter{relativedistance}{0}
                \setcounter{toremove}{\x}
                
                \foreach \k in {\j,...,1}{
                    \pgfmathsetmacro{\substepnu}{\mynu[\k]}
                    \pgfmathsetmacro{\substepmu}{\mymu[\k]}
                    \pgfmathsetmacro{\substepmubelow}{\mymu[\k-1]}
                    
                    \ifthenelse{\k=\j}
                        {\addtocounter{relativedistance}{\x-\substepnu}}
                        {\addtocounter{relativedistance}{\substepmu-\substepnu}}
    
                    \ifthenelse{\therelativedistance<0}
                        {\breakforeach}
                        {\addtocounter{toremove}{\substepmubelow}}
                }

                \ifthenelse{\x=0}
                {\setcounter{firstremove}{\thetoremove}}
                {}
                
                \node[fillnode,scale=#3] at (\thewidth-\thetoremove,\theheight) {};
            }

            \draw[treeedge] (\thewidth-\thetoremove,\theheight) -- (\thewidth-\thefirstremove,\theheight);
            
            \ifthenelse{\j=\thetotalheight}{}{
            \setcounter{parentheight}{\theheight}
            \setcounter{relativedistance}{0}
            \foreach \k in {\j,...,\thetotalheight}{
                \addtocounter{parentheight}{1}

                \pgfmathsetmacro{\substepnuabove}{\mynu[\k+1]}
                \pgfmathsetmacro{\substepmuabove}{\mymu[\k+1]}
                \addtocounter{relativedistance}{\substepnuabove-\substepmuabove}

                \ifthenelse{\therelativedistance>0}{}{\breakforeach} 
            }            
            \draw[treeedge] (\thewidth-\thetoremove,\theheight) -- (\thewidth-\thetoremove,\theparentheight);
            }
            
            \addtocounter{height}{1}
        }

 		\foreach \a/\b/\c/\d in {#4}{
			\ifthenelse{\equal{\c}{}}{}{
				\draw[\c,line width=.5pt](\a,\b) circle(.25);
			}
			\fill[\d,line width=.5pt](\a,\b) circle(.15);
 		}
		\foreach \a/\b/\c/\d/\e in {#5}{
			\draw[\d](\a,\b) node[scale=.9,anchor=\e]{\c};
		}
    \end{tikzpicture}
}

\newcommand{\altnuTree}[8]{
    \begin{tikzpicture}[scale=#4,
            fillnode/.style={fill=black,draw=black,circle,scale=1},
            circlenode/.style={draw=black,circle,scale=1},
            treeedge/.style={line width=0.9pt,black},
    ]
		\setcounter{valleyleft}{0}
		\setcounter{height}{0}
		\setcounter{previousheight}{0}
        \setcounter{totalheight}{-1}

        #8
        
        \edef\mynu{{#1}}
        \edef\mydelta{{#2}}
        \edef\mymu{{#3}}

            \foreach \step in {#1}{
                \addtocounter{totalheight}{1}
            }

            \foreach \j in {1,...,\thetotalheight}{    
                \pgfmathsetmacro{\stepnu}{\mynu[\j]}
                \pgfmathsetmacro{\stepdelta}{\mydelta[\j-1]}
                \addtocounter{valleyleft}{\stepnu-\stepdelta}
            }
            \setcounter{valley}{\thevalleyleft}
            \pgfmathsetmacro{\stepnuzero}{\mynu[0]}
            \addtocounter{valley}{\stepnuzero}
            \setcounter{widthzero}{\thevalley}

            \draw[\gray, dashed] (\thevalleyleft,\theheight) -- (\thevalley, \theheight);
            \foreach \i in  {\thevalleyleft,...,\thevalley}{
                \draw[\gray, dashed] (\i,\theheight) -- (\i,\theheight+1);
            }
            \addtocounter{height}{1}

            \foreach \j in {1,...,\thetotalheight}{
                \pgfmathsetmacro{\stepnu}{\mynu[\j]}
                \pgfmathsetmacro{\stepdelta}{\mydelta[\j-1]}
                \addtocounter{valleyleft}{\stepdelta-\stepnu}
                \addtocounter{valley}{\stepdelta}

                \ifthenelse{\j=\thetotalheight}{
                    \draw[\gray, dashed] (\thevalleyleft,\theheight) -- (\thevalley, \theheight);
                    }
                    {
                        \draw[\gray, dashed] (\thevalleyleft,\theheight) -- (\thevalley, \theheight);
                        \foreach \i in  {\thevalleyleft,...,\thevalley}{
                            \draw[\gray, dashed] (\i,\theheight) -- (\i,\theheight+1);
                        }
                    }
                    
                \addtocounter{height}{1}
            }

		\setcounter{height}{0}
  		\setcounter{width}{0}

            \pgfmathsetmacro{\stepnu}{\mynu[0]}
            \pgfmathsetmacro{\stepmu}{\mymu[0]}
            \addtocounter{width}{\thewidthzero} 

            \foreach \x in {0,...,\stepmu}{
                \node[fillnode,scale=#4] at (\thewidth-\x,\theheight) {};
            }
            
            \draw[treeedge] (\thewidth-\stepmu,\theheight) -- (\thewidth,\theheight);

            \ifthenelse{0=\thetotalheight}{}{
            \setcounter{parentheight}{\theheight}
            \setcounter{relativedistance}{0}
            \foreach \k in {0,...,\thetotalheight}{
                \addtocounter{parentheight}{1}
                
                \pgfmathsetmacro{\substepnuabove}{\mydelta[\k]} 
                \pgfmathsetmacro{\substepmuabove}{\mymu[\k+1]}
                \addtocounter{relativedistance}{\substepnuabove-\substepmuabove}
                
                \ifthenelse{\therelativedistance>0}{}{\breakforeach} 
            }            
            \draw[treeedge] (\thewidth-\stepmu,\theheight) -- (\thewidth-\stepmu,\theparentheight);
            }

            \addtocounter{height}{1}

        \foreach \j in {1,...,\thetotalheight}{

            \pgfmathsetmacro{\stepnu}{\mynu[\j]}
            \pgfmathsetmacro{\stepmu}{\mymu[\j]}
            \pgfmathsetmacro{\stepdelta}{\mydelta[\j-1]}
            \addtocounter{width}{\stepdelta} 

            \foreach \x in {0,...,\stepmu}{
                \setcounter{relativedistance}{0}
                \setcounter{toremove}{\x}
                
                \foreach \k in {\j,...,1}{
                    \pgfmathsetmacro{\substepnu}{\mydelta[\k-1]} 
                    \pgfmathsetmacro{\substepmu}{\mymu[\k]}
                    \pgfmathsetmacro{\substepmubelow}{\mymu[\k-1]}
                    
                    \ifthenelse{\k=\j}
                        {\addtocounter{relativedistance}{\x-\substepnu}}
                        {\addtocounter{relativedistance}{\substepmu-\substepnu}}
    
                    \ifthenelse{\therelativedistance<0}
                        {\breakforeach}
                        {\addtocounter{toremove}{\substepmubelow}}
                }

                \ifthenelse{\x=0}
                {\setcounter{firstremove}{\thetoremove}}
                {}
                
                \node[fillnode,scale=#4] at (\thewidth-\thetoremove,\theheight) {};
            }

            \draw[treeedge] (\thewidth-\thetoremove,\theheight) -- (\thewidth-\thefirstremove,\theheight);
            
            \ifthenelse{\j=\thetotalheight}{}{
            \setcounter{parentheight}{\theheight}
            \setcounter{relativedistance}{0}
            \foreach \k in {\j,...,\thetotalheight}{
                \addtocounter{parentheight}{1}

                \pgfmathsetmacro{\substepnuabove}{\mydelta[\k]} 
                \pgfmathsetmacro{\substepmuabove}{\mymu[\k+1]}
                \addtocounter{relativedistance}{\substepnuabove-\substepmuabove}

                \ifthenelse{\therelativedistance>0}{}{\breakforeach} 
            }            
            \draw[treeedge] (\thewidth-\thetoremove,\theheight) -- (\thewidth-\thetoremove,\theparentheight);
            }
            
            \addtocounter{height}{1}
        }

 		\foreach \a/\b/\c/\d in {#5}{
			\ifthenelse{\equal{\c}{}}{}{
				\draw[\c,line width=.5pt](\a,\b) circle(.25);
			}
			\fill[\d,line width=.5pt](\a,\b) circle(.15);
 		}
		\foreach \a/\b/\c/\d/\e in {#6}{
			\draw[\d](\a,\b) node[scale=.9,anchor=\e]{\c};
		}
    \end{tikzpicture}
}

\newcommand{\altnuGrid}[8]{
    \begin{tikzpicture}[scale=#4,
            fillnode/.style={fill=black,draw=black,circle,scale=1},
            circlenode/.style={draw=black,circle,scale=1},
            treeedge/.style={line width=0.9pt,black},
    ]
		\setcounter{valleyleft}{0}
		\setcounter{height}{0}
		\setcounter{previousheight}{0}
        \setcounter{totalheight}{-1}

        #8
        
        \edef\mynu{{#1}}
        \edef\mydelta{{#2}}

            \foreach \step in {#1}{
                \addtocounter{totalheight}{1}
            }

            \foreach \j in {1,...,\thetotalheight}{    
                \pgfmathsetmacro{\stepnu}{\mynu[\j]}
                \pgfmathsetmacro{\stepdelta}{\mydelta[\j-1]}
                \addtocounter{valleyleft}{\stepnu-\stepdelta}
            }
            \setcounter{valley}{\thevalleyleft}
            \pgfmathsetmacro{\stepnuzero}{\mynu[0]}
            \addtocounter{valley}{\stepnuzero}
            \setcounter{widthzero}{\thevalley}

            \draw[\gray, dashed] (\thevalleyleft,\theheight) -- (\thevalley, \theheight);
            \foreach \i in  {\thevalleyleft,...,\thevalley}{
                \draw[\gray, dashed] (\i,\theheight) -- (\i,\theheight+1);
            }
            \addtocounter{height}{1}

            \foreach \j in {1,...,\thetotalheight}{
                \pgfmathsetmacro{\stepnu}{\mynu[\j]}
                \pgfmathsetmacro{\stepdelta}{\mydelta[\j-1]}
                \addtocounter{valleyleft}{\stepdelta-\stepnu}
                \addtocounter{valley}{\stepdelta}

                \ifthenelse{\j=\thetotalheight}{
                    \draw[\gray, dashed] (\thevalleyleft,\theheight) -- (\thevalley, \theheight);
                    }
                    {
                        \draw[\gray, dashed] (\thevalleyleft,\theheight) -- (\thevalley, \theheight);
                        \foreach \i in  {\thevalleyleft,...,\thevalley}{
                            \draw[\gray, dashed] (\i,\theheight) -- (\i,\theheight+1);
                        }
                    }
                    
                \addtocounter{height}{1}
            }

 		\foreach \a/\b/\c/\d in {#5}{
			\ifthenelse{\equal{\c}{}}{}{
				\draw[\c,line width=.5pt](\a,\b) circle(.25);
			}
			\fill[\d,line width=.5pt](\a,\b) circle(.15);
 		}
		\foreach \a/\b/\c/\d/\e in {#6}{
			\draw[\d](\a,\b) node[scale=.9,anchor=\e]{\c};
		}
    \end{tikzpicture}
}

\title[A canonical realization of the alt $\nu$-associahedron]{A canonical realization of the \\ alt $\nu$-associahedron}

\author{Cesar Ceballos}
\address{TU Graz, Institut f\"ur Geometrie, Kopernikusgasse 24, 8010 Graz, Austria.}
\email{cesar.ceballos@tugraz.at}\thanks{The author was supported by the Austrian Science Fund FWF, grants P 33278 and I 5788.}
\date{\today}


\begin{document}

\begin{abstract}
    Given a lattice path $\nu$, the alt $\nu$-Tamari lattice is a partial order recently introduced by Ceballos and Chenevière, which generalizes the $\nu$-Tamari lattice and the $\nu$-Dyck lattice. All these posets are defined on the set of lattice paths that lie weakly above $\nu$, and posses a rich combinatorial structure. In this paper, we study the geometric structure of these posets. 
    We show that their Hasse diagram is the edge graph of a polytopal complex induced by a tropical hyperplane arrangement, which we call the alt $\nu$-associahedron. This generalizes the realization of $\nu$-associahedra by Ceballos, Padrol and Sarmiento. Our approach leads to an elegant construction, in terms of areas below lattice paths, which we call the canonical realization.
    Surprisingly, in the case of the classical associahedron, our canonical realization magically recovers Loday's ubiquitous realization, via a simple affine transformation.  
\end{abstract}

\maketitle

\tableofcontents

\section{Introduction}\label{sec_intro}

The $n$-dimensional associahedron is a mythical polytope which has been extensively studied in the literature~\cite{tamari_festschrift}. 
Its vertices correspond to plane binary trees with $n+1$ internal nodes, and two trees are connected by an edge if they are related by a tree rotation. 
The associahedron appeared for the first time in Dov Tamari's doctoral thesis in~1951, where illustrations of the associahedron up to dimension~3 are shown. 
However, general geometric realizations of the associahedron in higher dimensions were discovered only until the mid and late~1980's, via explicit constructions by Mark Haiman~\cite{haiman_associhedron_1984} and Carl Lee~\cite{lee_associhedron_1989}. 
Since then, many different constructions methods emerged. We refer to~\cite{ceballos_many_2015} for more on the history of some constructions. 

One of the most ubiquitous constructions of the associahedron is the one commonly known as Loday's associahedron~\cite{loday_realizations_2004}. 
The history of this construction goes back to Steven Shnider and Shlomo Sternberg~\cite{shnider_from_1993}, who described it in terms of explicit defining inequalities. 
Then, Jean-Louis Loday~\cite{loday_realizations_2004} described how to obtain nice vertex coordinates of this associahedron using combinatorics of binary trees.
Loday's coordinate description is very elegant and simple: Given a plane binary tree $T$, label each internal node by the number of leaves on its left times the number of leaves on its right. The coordinate $L(T)=(\ell_1,\dots,\ell_{n+1})$ is obtained by reading the labels of the $n+1$ internal nodes in \emph{in-order}\footnote{Recursively read the labels in in-order of the left descendant tree, then the label of the root, and then the labels of the right descendant tree.}. 
Figure~\ref{fig_coordinates_Loday_3D}, shows the coordinates $L(T)$ of two plane binary trees with 4 internal nodes. Taking the convex hull of the points $L(T)$ over all plane binary trees with $n+1$ internal nodes gives a geometric realization of the $n$-dimensional associahedron embedded in $\mathbb{R}^{n+1}$.
We refer to this realization as \defn{Loday's associahedron}. 
The 3-dimensional case is illustrated in~\Cref{fig_associahedron_Loday_3D}, where the vertices are labeled by their coordinates in $\mathbb{R}^4$, omitting parenthesis and commas for simplicity.
We refer to~\cite{pilaud_celebrating_2023} for a nice survey celebrating Loday's associahedron and  its far reaching influence in combinatorics, discrete geometry and algebra.

\begin{figure}[htb]
    \centering
    \input{figures/coordinates_Loday_3D}
    \caption{Example of Loday's coordinates of two plane binary trees.}
    \label{fig_coordinates_Loday_3D}
\end{figure}

\begin{figure}[htb]
    \centering
    \input{figures/associahedron_Loday_3D}
    \caption{Loday's 3-dimensional associahedron.}
    \label{fig_associahedron_Loday_3D}
\end{figure}

In this work, we unexpectedly discovered a new appearance of Loday's associahedron (via a simple affine transformation) in a much wider generality. Before going into details about the general framework, we would like to explain our construction in the special case of the classical associahedron to highlight the beauty of both constructions and their relation.

Each plane binary tree $T$ with $n+1$ internal nodes can be drawn uniquely on a staircase of size~$n+1$, as illustrated in~\Cref{fig_coordinates_canonical_3D}. More precisely, the nodes of the tree $T$ (internal nodes and leaves), are lattice points weakly above the path $(NE)^{n+1}$, where N stands for a north step and E for an east step. The root of the tree is always located at the top left corner, and every other node is connected by an edge to the node directly to its left if any, and to the node directly above it if any. The leaves of the tree are located on the diagonal of the staircase. This way of drawing plane binary trees is quite convenient and extremely useful.     

Our next goal is to describe another elegant way to provide coordinates to plane binary trees, providing another beautiful realization of the associahedron. 

Let $T$ be a plane binary tree with $n+1$ internal nodes drawn on a stair case of size $n+1$. Let $T_i$ be the unique path connecting the left most node at height $i$ to the root in the tree $T$, for $i=1,\dots,n$. The coordinate $C(T)=(c_n,\dots,c_1)$ of the tree~$T$ is defined by 
\begin{align}
    c_i = \area(T_i),
\end{align}
where $\area(T_i)$ is the number of boxes below the path $T_i$. See~\Cref{fig_coordinates_canonical_3D} for an illustration.
The convex hull of the points $C(T)=(c_n,\dots,c_1)$, over all plane binary trees with $n+1$ internal nodes, is a  realization of the $n$-dimensional associahedron which we call the \defn{canonical realization}.  
The 3-dimensional case is illustrated in~\Cref{fig_associahedron_canonical_3D}, where the vertices are labeled by their coordinates in $\mathbb{R}^3$, omitting parenthesis and commas for simplicity. 

\begin{figure}[htb]
    \centering
    \begin{tabular}{cc}

\nuTree{0,1,1,1,1}{0,0,1,0,3}{0.5}{}{-0.2/3/$3$/red/east,-0.2/2/$2$/blue/east,-0.2/1/$3$/green/east}{}{
\filldraw[nearly transparent, red] (3,3) -- (3,4) -- (0,4) -- (0,3);
\filldraw[nearly transparent, blue] (1,2) -- (1,4) -- (0,4) -- (0,2);
\filldraw[nearly transparent, green] (1,1) -- (1,4) -- (0,4) -- (0,1);
}     

     &  

\nuTree{0,1,1,1,1}{0,0,2,0,2}{0.5}{}{-0.2/3/$3$/red/east,-0.2/2/$0$/blue/east,-0.2/1/$1$/green/east}{}{
\filldraw[nearly transparent, red] (3,3) -- (3,4) -- (0,4) -- (0,3);
\filldraw[nearly transparent, green] (1,1) -- (1,2) -- (0,2) -- (0,1);
}
     
     \\ 
\rule{0pt}{3ex}    

$(3,2,3)$     
     & 
$(3,0,1)$
\end{tabular}
    \caption{Example of coordinates of two plane binary trees in our canonical realization.}
    \label{fig_coordinates_canonical_3D}
\end{figure}

\begin{figure}[htb]
    \centering
    \input{figures/associahedron_canonical_3D}
    \caption{Our canonical realization of the 3-dimensional associahedron.}
    \label{fig_associahedron_canonical_3D}
\end{figure}

Magically, Loday's associahedron and our canonical realization of the associahedron are related by the simple affine transformation
\begin{align}
  \varphi:\mathbb{R}^{n+1}&\to \mathbb{R}^n \\
  \varphi(\ell_1,\dots,\ell_{n+1})&=(c_n,\dots,c_1)
\end{align}
defined by
\begin{align}
    c_i=(\ell_1+\dots+\ell_i) - (1+\dots+i).
\end{align}
For instance, 
\begin{align*}
    \varphi(4,1,4,1) &= (3,2,3) \\
    \varphi(2,1,6,1) &= (3,0,1). 
\end{align*}
That is, the Loday coordinates $L(T)$ of the two trees $T$ in~\Cref{fig_coordinates_Loday_3D} are mapped to our canonical coordinates $C(T)$ for the same trees computed in~\Cref{fig_coordinates_canonical_3D}.

Our description of the canonical realization of the associahedron is just an example of a more general construction of alt $\nu$-associahedra, which is the main contribution of this paper. Our construction concerns generalizations of the classical Tamari lattice, which have been developed mostly because of their importance in representation theory and generalizations of the theory of diagonal harmonics. 

The story starts with Francois Bergeron's introduction of the $m$-Tamari lattice as a partial order on $m$-Dyck paths~\cite{bergeron_higher_2012}, motivated by conjectural connections between the number of intervals in the classical Tamari lattice~\cite{chapoton_tamari_intervals_2005} and the dimension of the alternating component of an $\mathfrak S_n$-module in the study of trivariate diagonal harmonics~\cite{haiman_conjectures_1994}. 
These lattices have remarkable enumerative properties and connections to representation theory, as shown in~\cite{bousquet_representation_2013,bousquet_number_2011}. 
Motivated by all these connections, 
Pr\'eville-Ratelle and Viennot~\cite{preville_nu_tamari_2017} introduced a generalization called the $\nu$-Tamari lattice, indexed by a lattice path $\nu$. 
A further generalization of the $\nu$-Tamari lattice, known as the alt $\nu$-Tamari lattice was recently introduced in~\cite{ceballos_cheneviere_linear_2024}.

The $\nu$-Tamari lattice turned out to have a very rich underlying geometric structure. Its Hasse diagram can be realized as the edge graph of a polytopal complex called the $\nu$-associahedron~\cite{ceballos_geometry_2019}.  
In this paper, we present the first known geometric realizations of the alt $\nu$-Tamari lattice as the edge graph of a polytopal complex that we call the alt $\nu$-associahedron. Our approach generalizes the results in~\cite{ceballos_geometry_2019}, and is based on techniques from tropical geometry.  

One of our main results is the following generalization of~\cite[Theorem~1.1]{ceballos_geometry_2019} for alt $\nu$-Tamari lattices. We refer to later sections for explanation on the terminology.  
\Cref{fig_nuTamari_nDyck_ENEENN,fig_alt_nu_tamari_ENEENN_delta_one} illustrate two examples, where the vertices are labelled by $\nu$-Dyck paths (see~\Cref{sec_altnuTamari}). Further 3D examples are presented in~\Cref{sec_3D_examples}.


\begin{theorem}[{Theorem~\ref{thm_noncrossing_triangulation}, Corollary~\ref{cor_fine_mixed_subdivision}, and Definition~\ref{def_U_associhedron}/Theorem~\ref{def_thm_U_associahedron}}]
Let $\nu$ be a lattice path from $(0,0)$ to (a,b), and $\delta$ be an increment vector with respect to $\nu$.
The Hasse diagram of the alt $\nu$-Tamari lattice $\altTam{\nu}{\delta}$ can be realized geometrically as:
\begin{enumerate}
    \item the dual of a regular triangulation of a subpolytope of the product of two simplices $\Delta_a\times \Delta_b$;
    \item the dual of a coherent fine mixed subdivision of a generalized permutahedron (in $\mathbb R^a$ and in~$\mathbb R^b$);
    \item the edge graph of a polyhedral complex induced by an arrangement of tropical hyperplanes (in $\mathbb{TP}^a \cong \mathbb{R}^a$ and in $\mathbb{TP}^b \cong \mathbb{R}^b$). 
\end{enumerate}
\end{theorem}

\begin{figure}[h]
    \begin{center}
        \input{figures/alt_nu_tam_ENEENN_delta2}
        \input{figures/alt_nu_tam_ENEENN_delta0}
    \end{center}
    \caption{The $\nu$-Tamari lattice and $\nu$-Dyck lattice for $\nu=ENEENN=(1,2,0,0)$. They are the alt~$\nu$-Tamari lattices $\altTam{\nu}{\delta}$ for $\delta=(2,0,0)$ and $\delta=(0,0,0)$, respectively.}
    \label{fig_nuTamari_nDyck_ENEENN}
\end{figure}

\begin{figure}[h]
    \begin{center}
        \input{figures/alt_nu_tam_ENEENN_delta1}
    \end{center}
    \caption{The alt $\nu$-Tamari lattice $\altTam{\nu}{\delta}$ for $\nu=ENEENN=(1,2,0,0)$ and $\delta=(1,0,0)$.}
    \label{fig_alt_nu_tamari_ENEENN_delta_one}
\end{figure}

\section{The alt $\nu$-Tamari lattice}\label{sec_altnuTamari}

    In this section, we recall the definition and some results about the alt $\nu$-Tamari lattices introduced by Ceballos and Chenevi\`ere in~\cite{ceballos_cheneviere_linear_2024}. 
    This family of lattices includes the $\nu$-Tamari lattice~\cite{preville_nu_tamari_2017} and the $\nu$-Dyck lattice as extreme cases.

\subsection{The alt $\nu$-Tamari lattice on paths}
    Let $\nu$ be a lattice path on the plane that starts at the origin and consists of a finite number of east and north unit steps. 
    Throughout the paper, we assume that this path ends at the coordinate~$(a,b)$, where~$a$ is the number of east steps and $b$ is the number of north steps of $\nu$. 
    Sometimes, we represent a lattice path $\nu$ by its sequence of letters~$E$ and $N$ for east and north steps, respectively. Alternatively, we also represent $\nu$ by a sequence of non negative integers~$(\nu_0, \nu_1, \dots, \nu_b)$ where $\nu_i$ is the number of east steps at height~$i$.
    For instance, the path $\nu=ENEENN$ can be represented by the sequence $\nu=(1,2,0,0)$. The vertices in~\Cref{fig_nuTamari_nDyck_ENEENN,fig_alt_nu_tamari_ENEENN_delta_one} are labeled using this convention.
    
    A \defn{$\nu$-path} $\mu$ is a lattice path using north and east steps, with the same endpoints as $\nu$, that is weakly above $\nu$.
    Alternatively, $\mu = (\mu_0, \dots, \mu_b)$ is a $\nu$-path if and only if $\sum_{i = 0}^{j} \mu_i \leq \sum_{i = 0}^{j} \nu_i $ for all $ 0 \leq j \leq b$,
    with equality holding for $j=b$.
     
    Let $\nu = (\nu_0, \dots, \nu_b)$ be a fixed path.
    We say that $\delta = (\delta_1, \dots, \delta_b) \in \mathbb{N}^b$ is an \defn{increment vector} with respect to $\nu$ if $\delta_i \leq \nu_i$ for all $ 1 \leq i \leq b $.	
    The alt $\nu$-Tamari lattice associated to $\delta$ is defined in terms of a notion of~$\delta$-altitude.

    We set the $\delta$-altitude of the initial lattice point of $\mu$ to be equal to zero,
    and declare that the $i$-th north step of $\mu$ increases the $\delta$-altitude by $\delta_i$ and an east step decreases the $\delta$-altitude by $1$.
    We denote by $\alt_\delta(q)$ the \defn{$\delta$-altitude} of a lattice point $q$ of $\mu$. In other words, if $q$ has coordinates $q=(i,j)$ then its $\delta$-altitude is $\alt_\delta(q)=(\delta_1+\dots+\delta_j) - i$. 
    
    Given a valley $p$ of $\mu$ (a point preceeded by an east step and followed by a north step), we let $q$ be the next lattice point of $\mu$ after $p$ such that $ \alt_\delta(q) =  \alt_\delta(p)$, and $\mu_{[p,q]}$ be the subpath of $\mu$ that starts at $p$ and ends at $q$.
    We denote by $\mu'$ the path obtained from $\mu$ by switching $ \mu_{[p,q]} $ with the east step~$E$ that precedes it.
    The \defn{$\delta$-rotation} of~$\mu$ at the valley $ p $ is defined to be $\mu'$, and we write $\mu \drot{\delta}  \mu'$. An example is illustrated in Figure~\ref{fig_delta_rotation_paths}.

    \begin{figure}[htb]
        \centering
        \input{figures/delta_rotation_paths}
        \caption{The $\delta$-rotation of a $\nu$-path for $\nu=(2,2,3,1,0)$ and $\delta=(1,2,0,0)$. Each node is labelled with its $\delta$-altitude.}
        \label{fig_delta_rotation_paths}
    \end{figure}

	\begin{definition}
        Let $\delta$ be an increment vector with respect to $\nu$.
		The \defn{alt $\nu$-Tamari poset} $ \altTam{\nu}{\delta} $ is the reflexive transitive closure of $\delta$-rotations on the set of $\nu$-paths.
	\end{definition}

    Three examples of alt $\nu$-Tamari lattices $\altTam{\nu}{\delta}$ for $\nu=ENEEN=(1,2,0)$ and $\delta=(0,0),(1,0)$ and $(2,0)$ are illustrated in \Cref{fig_altnu_lattices_ENEEN_paths}.

    \begin{figure}[htb]
        \begin{center}
            \input{figures/alt_nu_tam_ENEEN_delta0}
            \input{figures/alt_nu_tam_ENEEN_delta1}
            \input{figures/alt_nu_tam_ENEEN_delta2}
        \end{center}
        \caption{Examples of alt $\nu$-Tamari lattices $\altTam{\nu}{\delta}$ for $\nu=ENEEN=(1,2,0)$. Left: the $\nu$-Dyck lattice, for $\delta=(0,0)$. Middle: the lattice for $\delta=(1,0)$. Right: the $\nu$-Tamari lattice, for $\delta=(2,0)$. 
        }
        \label{fig_altnu_lattices_ENEEN_paths}
    \end{figure}

	\begin{remark} \label{rem:subposet}
        For a fixed path $\nu$, there are two extreme cases for the possible choices of increment vector $\delta$. 
        If $\delta_i = \nu_i$ for all $ 1 \leq i \leq b $, the alt $\nu$-Tamari lattice recovers the $\nu$-Tamari lattice introduced by Pr\'eville-Ratelle and Viennot in~\cite{preville_nu_tamari_2017}. 
        If~$\delta_i = 0$ for all $ 1 \leq i \leq b$, the alt $\nu$-Tamari lattice is the poset on $\nu$-paths whose covering relations correspond to changing a consecutive pair $EN$ by $NE$ (adding a box to the path). In this case, $\mu_1 \leq \mu_2$ if $\mu_1$ is weakly below $\mu_2$. This poset is known as the $\nu$-Dyck lattice.  
	\end{remark}

    The following statement summarizes some of the main results in~\cite{ceballos_cheneviere_linear_2024}.

    \begin{theorem}[{\cite{ceballos_cheneviere_linear_2024}}]\label{thm_cc_altTam_main}
        Let $\delta$ be an increment vector with respect to a lattice path $\nu$.
        \begin{enumerate}
            \item The alt $\nu$-Tamari poset $\altTam{\nu}{\delta}$ is a lattice. 
            \item The covering relations of $\altTam{\nu}{\delta}$ are exactly $\delta$-rotations.
            \item For a fixed $\nu$, all alt $\nu$-Tamari lattices have the same number of linear intervals of any length.\footnote{A linear interval $[P,Q]$ is an interval that is a chain $P=P_0\leq P_1 \leq \dots \leq P_\ell=Q$, and its length is the length $\ell$ of the chain.}
        \end{enumerate}
    \end{theorem}

\subsection{The $U$-Tamari lattice on trees}

    The proof of Theorem~\ref{thm_cc_altTam_main} in~\cite{ceballos_cheneviere_linear_2024} is based on an alternative description of the alt $\nu$-Tamari lattice in terms of certain binary trees which they call $(\delta,\nu)$-trees. These trees can be viewed as maximal collections of lattice points in certain set that respect a compatibility relation that we now recall.   

    We say that a sequence  positive integers $\uu=\{u_0,\dots,u_a\}$ is \defn{unimodular} if 
    \[
    u_0\leq \dots \leq u_k \geq \dots \geq u_a
    \]
    for some value $k$.
    Given a unimodular sequence $\uu$, we denote by $b+1=\max\{u_0,\dots,u_a\}$ the maximum value of the $u_i$s. 
    We label the \defn{positions} of the lattice points of an $a\times b$ rectangle by pairs \defn{$(i,j)$} for $0\leq i \leq a$ and $0\leq j \leq b$, such that the value $i$ increases from left to right and $j$ increases from top to bottom.
    We say that the point at position $(i,j)$ is in \defn{column} $i$ and \defn{row} $j$. 
    We choose this convention because it is convenient for our purposes; note that the position of a lattice point is not given by the usual coordinate system. 
    The \defn{height} of a point at position $(i,j)$ is defined as $b-j$.
    The \defn{stack set} $U$ of a unimodular sequence $\uu$ is the subset of lattice points inside the $a\times b$ rectangle with positions $(i,j)$ for $0\leq i \leq a$ and $0 \leq j \leq u_i-1$. In other words, we stack or line up $u_i$ lattice points in column $i$.\footnote{Our name stack set is motivated by the definition of stack polyominoes in~\cite{jonsson_generalized_2005}, which are obtained by replacing our lattice points by unit boxes.} 
    A stack set~$U$ is called \defn{left-aligned} if its corresponding sequence is weakly decreasing: $u_0\geq \dots \geq u_a$. 
    \Cref{fig_stack_set,fig_left_aligned} illustrate examples of these concepts. 
    
    \begin{figure}[htb]
        \centering
        \input{figures/positions_stackset}
            \caption{Left: lattice points in an $a\times b$ rectangle labeled by their position $(i,j)$, for $a=8$ and $b=3$. 
        Right: the stack set $U$ of the  sequence $\uu=\{1,1,2,2,4,3,3,2,1\}$.}
        \label{fig_stack_set}
    \end{figure}

    \begin{figure}[htb]
        \centering
        \input{figures/left_aligned_set}       
        \caption{The left aligned stack set $U$ for the sequence $\uu=\{4,3,3,2,2,2,1,1,1\}$.}
        \label{fig_left_aligned}
    \end{figure}    

    We say that two points~$p,q\in U$ are \defn{$U$-incompatible} if $p$ is strictly southwest or strictly northeast of~$q$, and all the lattice points in the smallest rectangle containing $p$ and $q$ belong to~$U$.  Otherwise, $p$ and~$q$ are said to be \defn{$U$-compatible}. 
    A \defn{$U$-tree} is a maximal collection of pairwise $U$-compatible elements in $U$.
    We can visualize $U$-trees as classical rooted binary trees by connecting consecutive nodes in the same row or column. 
    The vertex at the top-left corner of $U$ is compatible with everyone else, and is the \defn{root} of every $U$-tree. We often refer to the lattice points in a $U$-tree as~\defn{nodes}.
    An example of a~$U$-tree and the rotation operation that we now define is illustrated in~\Cref{fig_delta_rotation_trees}.

    \begin{figure}[htb]
        \centering
        \begin{tabular}{ccc}
\altnuTree{2,2,3,1,0}{1,2,0,0}{1,2,1,2,2}{0.7}
{3/1/blue/red}
{3/3/$p$//south east,
3/1/$q=p\llcorner r$//north east,
4/1/$r$//north west}
{}
{
\fill [fill=red, fill opacity=0.5] (3,1) rectangle (4,3);
}

&
&

\altnuTree{2,2,3,1,0}{1,2,0,0}{1,1,1,3,2}{0.7}
{4/3/blue/red}
{3/3/$p$//south east,
4/3/$q'=p\urcorner r$//south west,
4/1/$r$//north west}
{}
{
\fill [fill=red, fill opacity=0.5] (3,1) rectangle (4,3);
}

\end{tabular}
        \caption{The rotation of a $U$-tree.}
        \label{fig_delta_rotation_trees}
    \end{figure}

    Let $T$ be a $U$-tree and $p,r\in T$ be two elements which do not lie in the same row or same column. We denote by $p\square r$ the smallest rectangle containing~$p$ and~$r$, and write~$p\llcorner r$~(resp. $p\urcorner r$) for the lower left corner (resp. upper right corner) of~$p\square r$.  

    Let $p,q,r\in T$ be such that $q=p\llcorner r$ and no other elements besides $p,q,r$ lie in~$p\square r$. 
    The \defn{rotation} of~$T$ at $q$ is defined as the set $T'=\bigl(T\setminus\{q\})\cup\{q'\}$, where~\mbox{$q'=p\urcorner r$}.  
    As shown in~\cite{ceballos_cheneviere_linear_2024} (cf.~\cite[Lemma~2.10]{ceballos_nu_trees_2020}), the rotation of a $U$-tree is also a $U$-tree. 
    
    \begin{definition}
    Let $U$ be the stack set of a unimodular sequence $\uu=\{u_0,\dots,u_a\}$.
    The \defn{$U$-Tamari poset} $\Tam{U}$ is the reflexive transitive closure of rotations on $U$-trees. We also refer to this poset as the \defn{rotation poset of~$U$-trees}.  
    \end{definition}

    Three examples of the rotation poset of~$U$-trees for the stack sets $U$ corresponding to the unimodular sequences $\uu=\{2,2,3,3\}$, $\uu=\{2,3,3,2\}$ and $\uu=\{3,3,2,2\}$ are illustrated in \Cref{fig_altnu_lattices_ENEEN_trees}.

    \begin{figure}[htb]
        \begin{center}
            \input{figures/alt_nu_tam_ENEEN_delta0_trees}
            \input{figures/alt_nu_tam_ENEEN_delta1_trees}
            \input{figures/alt_nu_tam_ENEEN_delta2_trees}
        \end{center}
        \caption{Examples of the rotation poset of~$U$-trees for the stack sets $U$ with unimodular sequences $\uu=\{2,2,3,3\}$, $\uu=\{2,3,3,2\}$ and $\uu=\{3,3,2,2\}$ 
        }
        \label{fig_altnu_lattices_ENEEN_trees}
    \end{figure}    

    In~\cite{ceballos_cheneviere_linear_2024}, Ceballos and Chenevière showed that the alt $\nu$-Tamari lattice $\altTam{\nu}{\delta}$ can be described as a rotation poset of $U$-trees for some specific stack set $U$. Our purpose now is to recall the construction of the set $U$ in terms of $\delta$ and $\nu$.  

    \begin{definition}
    Let $\delta$ be and increment vector with respect to $\nu$ (i.e. $\delta_i\leq \nu_i$). We define the \defn{stack set~$U_{\delta,\nu}$} as the collection of lattice points that are weakly above the path
    \begin{align}\label{path_of_delta_nu}
    (E^{\nu_b-\delta_b} S \dots E^{\nu_1-\delta_1} S) \quad 
    E^{\nu_0} \quad 
    (N E^{\delta_1}  \dots N E^{\delta_b})        
    \end{align}
    in the smallest rectangle containing it. The letters $E,S,N$ stand for east, south and north steps, respectively. 
    An $U_{\delta,\nu}$-tree is called a \defn{$(\delta,\nu)$-tree}.      
    \end{definition}

    \begin{example}\label{ex_stackset_from_delta}
        Let $\nu=NEENEEENEEE=(0,2,3,3)$. The stack set $U_{\delta^{\max},\nu}$ for $\delta^{\max}=(2,3,3)$ is the left aligned stack set illustrated in~\Cref{fig_left_aligned}. 
        The stack sets $U=U_{\delta,\nu}$ for $\delta=(2,1,1)$ and $U'=U_{\delta',\nu}$ for $\delta'=(1,2,2)$ are shown on the left and right of~\Cref{fig_permutation_unimodular_sequence}, respectively.
    \end{example}

    \begin{figure}[htb]
        \centering
        \input{figures/stack_sets}
        \caption{Two examples of stack sets for $\nu=(0,2,3,3)$. Left: the stack set $U=U_{\delta,\nu}$ for $\delta=(2,1,1)$. Right: the stack set $U'=U_{\delta',\nu}$ for $\delta'=(1,2,2)$.}
        \label{fig_permutation_unimodular_sequence}
    \end{figure}
    
    \begin{theorem}[{\cite{ceballos_cheneviere_linear_2024}}]\label{thm_cc_isomorphic_lattices}
        The alt $\nu$-Tamari lattice $\altTam{\nu}{\delta}$ and the 
        rotation poset of $(\delta,\nu)$-trees are isomorphic:
        \[
        \altTam{\nu}{\delta} \cong \Tam{U_{\delta,\nu}}.
        \]
    \end{theorem}

    The alt $\nu$-Tamari lattices in Figure~\ref{fig_altnu_lattices_ENEEN_paths} correspond to the rotation posets of $(\delta,\nu)$-trees in Figure~\ref{fig_altnu_lattices_ENEEN_trees}. 
    See~\Cref{fig_associahedron2D_paths_trees,associahedronENENEN_paths_trees,fig_Fuss_associahedron2D_m2,fig_Fuss_associahedron2D_m3,fig_alt_associahedronEENEEN_delta012} for further examples.

    \begin{remark}\label{rem_shape_vs_pair}
        Note that the number of east steps at height $i$ in the path defined by~\eqref{path_of_delta_nu} is equal to~$\nu_i$ (which is equal to $\nu_i-\delta_i+\delta_i$ for $i\neq 0$ and $\nu_0$ for $i=0$). 
        Therefore, every stack set $U$ corresponds to a unique pair $(\delta,\nu)$, where $\delta$ is an increment vector with respect to some lattice path $\nu$.
        We denote by~\defn{$(\delta_U,\nu_U)$} the unique pair associated to $U$.

        Furthermore, $\nu$ starts at $(0,0)$ and ends at $(a,b)$ if and only if the smallest rectangle containing $U_{\delta,\nu}$ is an $a\times b$ rectangle.        
    \end{remark}

    \begin{remark}\label{rem_shape_permutation_columns}
        If $\delta$ and $\delta'$ are two increment vectors with respect to $\nu$, then the corresponding shapes $U=U_{\delta,\nu}$ and $U'=U_{\delta',\nu}$ can be obtained from each other by some permutation of the columns. More precisely, if~$\uu=\{u_0,\dots ,u_a\}$ and~\mbox{$\uu'=\{u_0',\dots,u_a'\}$} are the unimodular sequences encoding the sizes of the columns of $U$ and $U'$, then~$u_i = u_{\sigma(i)}'$ for the permutation $\sigma$ of $\{0,\dots , a\}$ that sends the position of the $k$th appearance of a value $\ell$ in $\uu$ to the position of the $k$th appearance of $\ell$ in $\uu'$.

        Following~\Cref{ex_stackset_from_delta}, consider the path $\nu=NEENEEENEEE=(0,2,3,3)$ and the increment vectors $\delta=(2,1,1)$ and $\delta'=(1,2,2)$. The corresponding stack sets $U=U_{\delta,\nu}$ and $U'=U_{\delta',\nu}$, which are illustrated in~\Cref{fig_permutation_unimodular_sequence}, have unimodular sequences 
        \begin{align*}
            \uu=\{u_0,\dots, u_8\}&=\{1,1,2,2,4,3,3,2,1\}, \\
            \uu'=\{u_0',\dots, u_8'\}&=\{1,2,3,4,3,2,2,1,1\}.
        \end{align*}
        
        The permutation $\sigma$ on the set $\{0,\dots,8\}$ that transform one sequence to the other, moving the position of the $k$th appearance of a value $\ell$ in $\uu$ to the position of the $k$th appearance of $\ell$ in $\uu'$ is
        \[
        \begin{array}{cccccccccc}
            &0&1&2&3&4&5&6&7&8 \\
            
            \sigma=&0&7&1&5&3&2&4&6&8.
        \end{array}          
        \]
        
        For instance, $u_0=u_1=u_8=1$ are mapped to $u_0'=u_7'=u_8'=1$, which means that $\sigma(0)=0$, $\sigma(1)=7$ and $\sigma(8)=8$.  
        Similarly, $u_2=u_3=u_7=2$ are mapped to $u_1'=u_5'=u_6'=2$, the next $u_5=u_6=3$ are mapped to $u_2'=u_4'=3$, and $u_4=4$ is mapped to $u_3'=4$.  
    \end{remark}

\subsection{The isomorphism between both lattices}
    The isomorphism giving rise to Theorem~\ref{thm_cc_isomorphic_lattices} is beautiful and simple~\cite{ceballos_cheneviere_linear_2024}. It is a slight modification of the right-flushing bijection described by Ceballos, Padrol, and Sarmiento in~\cite{ceballos_nu_trees_2020}. 
    
    Given a~$\nu$-path~$\mu=(\mu_0,\mu_1,\dots,\mu_a)$, 
    there is a unique $(\delta,\nu)$-tree $T$ that has $\mu_i+1$ nodes at height $i$. 
    It can be constructed via the \defn{right flushing bijection} $\rightflushing{\delta}{\nu}$, which recursively adds  nodes to the tree from right to left, from bottom to top, avoiding forbidden positions. A position is forbidden if it is located above a previously added point which is not the left most node of its row. 
    An example is illustrated in~\Cref{fig_right_flushing}, where the forbidden positions are those belonging to the wiggly lines.

    \begin{figure}[htb]
        \centering
        \input{figures/right_flushing}
        \caption{The right flushing bijection $\rightflushing{\delta}{\nu}$ for $\nu=(2,2,3,1,0)$ and $\delta=(1,2,0,0)$.}
        \label{fig_right_flushing}
    \end{figure}

    \begin{proposition}[{\cite{ceballos_nu_trees_2020,ceballos_cheneviere_linear_2024}}]
    \label{prop_right_flushing_bijection}
        The right flushing $\rightflushing{\delta}{\nu}$ is a bijection between the set of $\nu$-paths and the set of $(\delta,\nu)$-trees. It sends $\delta$-rotations on $\nu$-paths to rotations on~$(\delta,\nu)$-trees. Thus, it is a poset isomorphism.
    \end{proposition}

    The inverse of the right flushing bijection is called the \defn{left flushing bijection}. It can be described similarly, by recursively adding points from left to right, from bottom to top, forbidding positions located above previously added points which are not the right most node of their row. The resulting set of points is by construction the set of lattice points of a lattice path.

\section{The $U$-Tamari complex}

    The description of alt $\nu$-Tamari lattices in terms on $(\delta,\nu)$-trees reveals a richer structure behind the poset, which is encoded as a simplicial complex. 

    \begin{definition}
    Let $U$ be the stack set of a unimodular sequence $\uu=\{u_0,\dots,u_a\}$. The \defn{$U$-Tamari complex $\TamComplex{U}$} is the simplicial complex of pairwise $U$-compatible subsets of~$U$.  
    The dimension of a face $I\in \TamComplex{U}$ is $\dim(I)=|I|-1$. 
    The facets (maximal dimensional faces) are the $U$-trees.
    \end{definition}

    \begin{lemma}
        The $U$-Tamari complex $\TamComplex{U}$ is a pure simplicial complex of dimension~\mbox{$a+b$}, where $a$ and $b$ are the lengths of the two sides of the smallest rectangle containing~$U$.  
    \end{lemma}

    \begin{proof}
        Assume that $U$ fits in an $a\times b$ rectangle. By Remark~\ref{rem_shape_vs_pair}, the set $U$ corresponds to a pair $(\delta,\nu)$ where $\delta$ is an increment vector with respect to a lattice path $\nu$ that starts at $(0,0)$ and ends at~$(a,b)$. 
        By Proposition~\ref{prop_right_flushing_bijection}, $U$-trees correspond to $\nu$-paths under the right flushing bijection. Since every $\nu$-path has $a+b+1$ lattice points, so does every $U$-tree. Thus,~$\TamComplex{U}$ is pure of dimension~$a+b$. 
    \end{proof}

    \begin{lemma}[{\cite[Lemma~12]{ceballos_nu_trees_2020}}]
    \label{lem_flips_equal_rotation}
        Two $U$-trees differ by a single element (i.e. the two facets intersect in a codimension 1 face) if and only if they are related by a rotation.
    \end{lemma}

    \begin{proof}
        This statement was proven in~\cite[Lemma~12]{ceballos_nu_trees_2020} in the case where $U$ is left aligned (which corresponds the classical $\nu$-Tamari case). Assume $U$ is not left aligned. Let $\check{U}\supseteq U$ be the smallest left aligned set containing $U$ in the $a\times b$ rectangle. $U$-trees are exactly~\mbox{$\check{U}$-trees} that are contained in $U$. Since the result holds for $\check{U}$-trees, it holds for $U$-trees.   
    \end{proof}
    
    The \defn{boundary} of $\TamComplex{U}$ is the simplicial complex induced by codimension 1 faces that are contained in exactly one facet, see Figure~\ref{fig_boundary_face} for an example. The faces that are not contained in the boundary are called \defn{interior}. Figure~\ref{fig_interior_faces} shows all interior faces for a stacked set $U$. They are drawn labelling the faces of a polytopal complex  which we call the $U$-associahedron further below, see~\Cref{def_thm_U_associahedron} and~\Cref{sec_tropical,sec_Uassociahedron}.

    \begin{figure}[htb]
    \begin{center}
        \input{figures/boundary_face}
    \end{center}
    \caption{A codimension 1 boundary face of $\TamComplex{U}$. 
    }
    \label{fig_boundary_face}
\end{figure}  

\begin{figure}[htb]
    \begin{center}
        \input{figures/interior_faces}
    \end{center}
    \caption{Interior faces of $\TamComplex{U}$ for the stack set $U$ with unimodular sequence $\uu=\{2,3,3,2\}$. 
    }
    \label{fig_interior_faces}
\end{figure}

\begin{remark}[Subword complexes]
    In~\cite[Section~5]{ceballos_nu_trees_2020}, it was shown that the $\nu$-Tamari complex is isomorphic to a well chosen subword complex $\mathcal{S}(Q_{\nu},w_{\nu})$ of type $A$, and that the $\nu$-Tamari lattice is isomorphic to its increasing flip poset.\footnote{We refer to~\cite{ceballos_nu_trees_2020} fro clarification on the subword complex terminology.} The same proof of~\cite[Section~5]{ceballos_nu_trees_2020} works in our more general~case. More precisely, For each pair $(\nu,\delta)$ we can find a word $Q_{\nu,\delta}$ and an element $w_{\nu,\delta}$ such that
    \begin{enumerate}
        \item the alt $\nu$-Tamari complex $\altTamComplex{\nu}{\delta}$ is isomorphic to the subword complex $\mathcal{S}(Q_{\delta,\nu},w_{\delta,\nu})$, and
        \item the alt $\nu$-Tamari lattice $\altTam{\nu}{\delta}$ is isomorphic to the increasing flip poset of $\mathcal{S}(Q_{\delta,\nu},w_{\delta,\nu})$.
    \end{enumerate}
    The word $Q_{\delta,\nu}$ and element $w_{\delta,\nu}$ can be obtained as follows. 
    For a lattice point $p$ in $U=U_{\delta,\nu}$, denote by $d(p)$ the lattice distance from $p$ to the upper left corner of $U$. 
    We label each lattice point $p\in U$ by the transposition $s_{d(p)+1}$. Define $Q_{\delta,\nu}$ as the word obtained by reading the labels in $U$ from left to right, from bottom to top. The element $w_{\delta_\nu}$ is the product of the transpositions in $Q_{\delta,\nu}$ corresponding to the complement of any alt $\nu$-tree (the key is that this is independent of the choice of the tree). This is illustrated for an example in~\Cref{fig_subword_complex}. The result is that $(\delta,\nu)$-trees correspond exactly to the complements of reduced expressions of $w_{\delta,\nu}$ in $Q_{\delta,\nu}$. 
    
    One immediate consequence of the subword complex approach is that alt $\nu$-Tamari complexes are vertex decomposable, and hence shellable~\cite{knutson_miller}. In particular, any linear extension of the alt $\nu$-Tamari lattice (or its opposite lattice) induces a shelling order on the alt $\nu$-Tamari complex.   
    
\end{remark}

\begin{figure}[h]
    \begin{center}
        \input{figures/subword_complex}
    \end{center}
    \caption{The alt $\nu$-Tamari complex as a well chosen subword complex.}
    \label{fig_subword_complex}
\end{figure}

\section{Geometric realizations}

    The main objective of this section is to provide geometric realizations of the $U$-Tamari complex and the $U$-Tamari lattice. In the case where $U$ is left aligned,
    our geometric realizations are exactly the same as the geometric realizations for $\nu$-Tamari lattices presented in~\cite{ceballos_geometry_2019}. 
    As we will see, the same ideas from~\cite{ceballos_geometry_2019}, adapted to our context and presented in a different language, work for any stack set $U$. 

    A big advantage of the presentation in this paper is that the combinatorial objects that we use are simpler to visualize (compare the drawing of a $U$-tree with the drawing of an $(I,J)$-tree in~\cite{ceballos_geometry_2019}).
    In addition, we present a new extremely simple and elegant realization of the $U$-associahedron in Section~\ref{sec_cononial_realization}.

\subsection{Triangulation of the polytope  $\subpolytope{U}$}\label{sec_triangulation}

    We start by presenting a geometric realization of the $U$-Tamari complex as a regular triangulation of a subpolytope~$\subpolytope{U}$ of the Cartesian product of two simplices~$\Delta_a\times \Delta_b$. 

    Let $U$ be the stack set associated to a unimodular sequence $\uu=(u_0,\dots,u_a)$ fitting in an $a\times b$ rectangle (smallest rectangle containing $U$, of width $a$ and height $b$). Recall that the position of the top left corner in $U$ is set to be $(0,0)$. The position of any other point in $U$ is a pair $(i,j)$,  where $i$ increases from left to right and $j$ increases from top to bottom.   

    We consider the product of two simplices
    \[
    \Delta_a \times \Delta_b = 
    \conv\{
    (e_i,e_j): 0\leq i \leq a,\ 0\leq j \leq b
    \}
    \subseteq \mathbb{R}^{a+1}\times \mathbb{R}^{b+1},
    \]
    where $e_0,\dots,e_a$ (resp. $e_0,\dots , e_b$) are the standard basis vectors of $\mathbb R^{a+1}$ (resp.~$\mathbb R^{b+1}$).
    We define the \defn{polytope $\subpolytope{U}$} as the subpolytope 
    \[
    \subpolytope{U} = 
    \conv\{
    (e_i,e_j): (i,j)\in U
    \}
    \subseteq \Delta_a \times \Delta_b.
    \]
    Both polytopes have dimension $a+b$.

    Recall that two lattice polytopes \(P,Q \subset \mathbb{R}^n\) are called \defn{integrally equivalent} if there exists an affine lattice automorphism
\(
x \mapsto Ax+b,
\)
where \(A \in \mathrm{GL}_n(\mathbb{Z})\) (i.e. A is an integer matrix with determinant~$\pm 1$) and \(b \in \mathbb{Z}^n\), such that
\[
Q = AP+b.
\]
In particular, integrally equivalent polytopes have the same normalized (and Euclidean) volume.

    \begin{lemma}\label{lem_changing_U}
        Let $\delta$ and $\delta'$ be two increment vectors with respect to a lattice path~$\nu$ that starts at $(0,0)$ and ends at $(a,b)$. 
        Let $U=U_{\delta,\nu}$ and $U'=U_{\delta',\nu}$ be their corresponding stack sets.  
        The subpolytopes~$\subpolytope{U}$ and $\subpolytope{U'}$ are integrally equivalent: they are related by
        \[
        \subpolytope{U'} = \widetilde \sigma (\subpolytope{U}),
        \]
        for the liner map 
        \begin{align*}
            \widetilde \sigma:\mathbb{R}^{a+1}\times \mathbb{R}^{b+1} & \rightarrow \mathbb{R}^{a+1}\times \mathbb{R}^{b+1} \\
            (e_i,e_j) &\rightarrow (e_{\sigma(i)},e_j)
        \end{align*}
        induced by the permutation $\sigma$ on the set $\{0,\dots,a\}$ from Remark~\ref{rem_shape_permutation_columns}. 
    \end{lemma}
    \begin{proof}
        By Remark~\ref{rem_shape_permutation_columns}, the two stack sets $U$ and $U'$ are obtained from each other by a permutation $\sigma$ of their columns. This permutation satisfies that 
        $(i,j)$ is a lattice point in~$U$ if and only if  
        $(\sigma (i),j)$ is a lattice point in $U'$. 
        The results follows.
    \end{proof}

    Recall that a \defn{subdivision} of a polytope is a partition of it into subpolytopes such that the intersection of any two subpolytopes is a (possibly empty) face of both. The subdivision is called a \defn{triangulation} if every subpolytope is a simplex. We refer to~\cite{loera_triangulations_2010} for further background on triangulations of polytopes.  

    \begin{lemma}\label{lem_number_max_cells}
        The number of maximal dimensional simplices in every triangulation of $\subpolytope{U}$ is equal to the number of $U$-trees. 
    \end{lemma}

    \begin{proof}
        Let $\nu=\nu_U$ be the lattice path corresponding to the stack set $U$ as in Remark~\ref{rem_shape_vs_pair}, and let $\delta=\delta_U$ be the corresponding increment vector. Let $\delta'=(\nu_1,\dots,\nu_b)$ be the maximum possible choice for an increment vector with respect to $\nu$. The set $U'=U_{\delta',\nu}$ is the set of lattice points weakly above $\nu$. 
        By Lemma~\ref{lem_changing_U}, triangulations of $\subpolytope{U}$ and $\subpolytope{U'}$ are in correspondence with each other via the transformation~$\widetilde \sigma$, and therefore have the same number of maximal dimensional simplices. 

        On the other hand, it is well known that the product of two simplices is a \emph{unimodular polytope}, see~\cite[Section~6.2.2]{loera_triangulations_2010}; In particular, the number of $k$-dimensional faces in every triangulation of a product of two simplices is fixed for every value $k$. And this also holds for the subpolytopes $\subpolytope{U}$ and~$\subpolytope{U'}$. 
        
        Therefore, we just need to show that there is one triangulation of $\subpolytope{U'}$ with the desired number of maximal dimensional simplices. 
        Consider the restriction of the \emph{stair case triangulation} of $\Delta_a\times \Delta_b$ to the subpolytope $\subpolytope{U'}$~\cite[Section~6.2.3]{loera_triangulations_2010}. This triangulation has one maximal dimensional simplex $I_\mu$ for each $\nu$-path $\mu$, given by
        \[
        I_{\mu} = \conv \{
        (e_i,e_j) : (i,j) \text{  is a lattice point of } \mu
        \}.
        \]
        Since the number of $\nu$-paths is equal to the number of $U$-trees by Propostion~\ref{prop_right_flushing_bijection}, the result follows. 
    \end{proof}

    We can think of a triangulation of $\subpolytope{U}$ as a collection of subsets of $U$.
    Our goal now is to show that the faces of the $U$-Tamari complex index the cells of a triangulation of $\subpolytope{U}$. We will achieve this via a regular triangulation.
    
    Recall that a \defn{regular subdivision $\subpolytope{U}^h$} of $\subpolytope{U}$ is the projection of the \defn{upper envelope} of the polytope 
    \[
    \conv\{(e_i,e_j,h(i,j)): (e_i,e_j)\in \subpolytope{U}\},
    \]
    which is obtained by lifting the vertices of $\subpolytope{U}$ by a \defn{height function} $h:U\rightarrow \mathbb R$. 
    If the height function is sufficiently generic, all the cells of the subdivision turn out to be simplices. The resulting triangulation is called a \defn{regular triangulation}. Theorem~\ref{thm_noncrossing_triangulation} below characterizes the height functions for which the induced triangulation is the~\mbox{$U$-Tamari} complex.

    \begin{definition}\label{def_height_function}
    The height function $h:U\rightarrow \mathbb R$ is called \defn{non-crossing}, if for every~$i<i'$ and $j<j'$ such that~$(i,j),(i,j'),(i',j'),(i',j)\in U$, we have that 
    \begin{align}\label{eq_noncrossing}
        h(i,j)+h(i',j') > h(i,j')+h(i',j).
    \end{align}
    In other words, $h$ is a non-crossing height function if for every rectangle with corners in $U$, the sum of the heights of the upper-left and lower-right corners is bigger than the sum of the heights of the lower-left and upper-right corners. 
    \end{definition}

    \begin{center}
        \begin{tikzpicture}
            \node (p) at (0,1) {$h(i,j)$};
            \node (r) at (2,0) {h$(i',j')$};
            \draw (p) -- (r);
        \end{tikzpicture}     
        \quad
        \begin{tikzpicture}
            \node at (0,0.7) {$>$};
            \node at (0,0) {};
        \end{tikzpicture}
        \quad 
        \begin{tikzpicture}
            \node (q) at (0,0) {$h(i,j')$};
            \node (q') at (2,1) {$h(i',j)$};
            \draw (q) -- (q');
        \end{tikzpicture}     
    \end{center}

    \begin{example}[The canonical height function]\label{ex_canonical_height_function}
        The map defined by $h(i,j)=ij$ is a non-crossing height function for any unimodular shape $U$. We call it the \defn{canonical height function}. The properties of the induced geometric realizations will be analyzed in Section~\ref{sec_cononial_realization}. 
    \end{example}

    \begin{theorem}[{Cf. \cite[Proposition~2.6]{ceballos_geometry_2019}}]\label{thm_noncrossing_triangulation}
        Let $U$ be the stack set of a unimodular sequence $\uu=(u_0,\dots,u_a)$. The following hold:
        \begin{enumerate}
            \item The $U$-Tamari complex $\TamComplex{U}$ indexes the simplices of a flag regular triangulation of $\subpolytope{U}$. Its maximal simplices are given by $U$-trees. Its dual graph is the Hasse diagram of the $U$-Tamari lattice~$\Tam{U}$.
            \item The regular subdivision $\subpolytope{U}^h$ induced by a height function $h:U\rightarrow \mathbb R$ is equal to $\TamComplex{U}$ if and only if $h$ is non-crossing.\label{item_two_thm_noncrossing}
        \end{enumerate}        
    \end{theorem}

    \begin{proof}
        Let $h:U\rightarrow \mathbb R$ be a height function. It suffices to show part~\eqref{item_two_thm_noncrossing}. 
    
        The restriction of $\subpolytope{U}^h$ to the rectangle formed by four points $(i,j),(i,j'),(i',j'),(i',j)\in U$ is one of the following three possible subdivisions:

        \begin{center}
        \begin{tikzpicture}
            \node (p) at (0,1) {$(i,j)$};
            \node (q) at (0,0) {$(i,j')$};
            \node (r) at (2,0) {$(i',j')$};
            \node (q') at (2,1) {$(i',j)$};
            \draw (p) -- (q) -- (r) -- (q') -- (p);
            \draw (p) -- (r);
        \end{tikzpicture}            
        \quad 
        \begin{tikzpicture}
            \node (p) at (0,1) {$(i,j)$};
            \node (q) at (0,0) {$(i,j')$};
            \node (r) at (2,0) {$(i',j')$};
            \node (q') at (2,1) {$(i',j)$};
            \draw (p) -- (q) -- (r) -- (q') -- (p);
            \draw (q) -- (q');
        \end{tikzpicture}     
        \quad
                \begin{tikzpicture}
            \node (p) at (0,1) {$(i,j)$};
            \node (q) at (0,0) {$(i,j')$};
            \node (r) at (2,0) {$(i',j')$};
            \node (q') at (2,1) {$(i',j)$};
            \draw (p) -- (q) -- (r) -- (q') -- (p);
        \end{tikzpicture}            
        \end{center}

        The first case is obtained when Equation~\eqref{eq_noncrossing} is satisfied ($h$ is non-crossing). The second and third cases are obtained when $>$ is replaced by $<$ or $=$ in Equation~\eqref{eq_noncrossing}, respectively.

        Assume that $\subpolytope{U}^h$ is a triangulation equal to $\Tam{U}$. 
        Since its is a triangulation, case 3 can not appear. 
        Since $(i,j')$ and $(i',j)$ are $U$-incompatible, case 2 can not appear either. 
        Therefore case 1 is forced to happen, in which case Equation~\eqref{eq_noncrossing} is satisfied, and so $h$ is non-crossing.

        Now assume that $h$ is non-crossing and prove that $\subpolytope{U}^h$ is equal to $\Tam{U}$. 
        Let~$A$ be a subset of~$U$ that represents the vertices of a maximal dimensional cell in~$\subpolytope{U}^h$. Since $h$ is non-crossing, $A$ is a pairwise $U$-compatible subset of $U$ (otherwise the diagonal edge in case 2 would appear, which is a contradiction). Since the dimension of the cell corresponding to $A$ is $a+b$, then $A$ has at least $a+b+1$ elements. Therefore $A$ is a $U$-tree and has exactly $a+b+1$ elements (since all maximal $U$-compatible elements have exactly $a+b+1$ elements). Furthermore, the cell of~$A$ is a simplex, and so $\subpolytope{U}$ is a triangulation.

        We have proven that if $h$ is non-crossing then $\subpolytope{U}^h$ is a triangulation whose maximal simplices are $U$-trees. It remains to show that every $U$-tree is a maximal simplex in this triangulation. This is achieved by counting the number of simplices in this triangulation, and seeing that it coincides with the number of $U$-trees. This was proven in Lemma~\ref{lem_number_max_cells}.        
    \end{proof}
    
    \begin{remark}
        In the literature, it is common to define a regular triangulation of a polytope by projecting the lower envelope of a lifting of its vertices.
        Contrary to how it is done in~\cite{ceballos_geometry_2019},
        we use the \emph{upper envelope} in this paper for convenience. 
        This will be reflected in Section~\ref{sec_tropical}, where we consider tropical hyperplanes defined by the $\min$ operation instead of $\max$. 
    \end{remark}

    \begin{remark}
        In the case where $U$ is the set of lattice points of an $a\times b$ rectangle, a non-crossing height function $h$ with respect to $U$ is just a $(b+1)\times (a+1)$ matrix such that for every quadrangle, the sum of the entries of the upper-left and lower-right corners is bigger than the sum of the entries of the lower-left and upper-right corners. If we allow equality in the defining Equation~\eqref{eq_noncrossing} (that is, replace $>$ by $\geq$), such a matrix is known as a \defn{Monge matrix} in combinatorial optimization~\cite{bukard_monge_1996}.\footnote{The definition of a Monge matrix uses the inequality $\leq$ instead of $\geq$. This small difference can be fixed by reversing the order of the rows, reversing the order of the columns, or by multiplying the matrix by $-1$.} There has been a substantial amount of research about Monge properties and their applications (see~\cite{bukard_monge_1996} and the references therein). It would be interesting to investigate, whether the results of this paper have some relation with problems in combinatorial optimization. 
        
        In particular,  the $U$-associahedron introduced in Section~\ref{sec_tropical} could play an interesting role in some combinatorial optimization models. In the case where $U$ is the set of lattice points in an $a\times b$ rectangle, the vertices of the $U$-associahedron encode partitions fitting in the rectangle, and the edges connect partitions that differ by adding or removing one box. The geometric realizations of this polytopal complex in Sections~\ref{sec_tropical} and~\ref{sec_cononial_realization} might have some connections to optimization problems.  
    \end{remark}
    
\subsection{Fine mixed subdivision of the polytope $\genPerm{U}$}
\label{sec_fine_mixed_subdivision}
    
    The triangulation of the polytope $\subpolytope{U}$ in Theorem~\ref{thm_noncrossing_triangulation} is the central geometric realization of the $U$-Tamari complex in this paper. Although it is very powerful, it has the disadvantage of having a big dimension which is hard to visualize. 
    One can overcome this disadvantage by applying the Cayley trick from~\cite{huber_rambau_santos_cayleytrick_2000}, which allows us to represent a high dimensional triangulation in fewer dimensions.

    As above, let $U$ be the stack set of a unimodular sequence $\uu=(u_0,\dots, u_a)$ and assume that the smallest rectangle containing it is an $a\times b$ rectangle. Given a subset $J\subseteq \{0,1,\dots, b\}$, define 
    \[
    \Delta_J= \conv\{e_j: j \in J\}\subseteq \mathbb{R}^{b+1}.
    \]
    The \defn{polytope $\genPerm{U}$} is the Minkowski sum 
    \begin{align}\label{eq_gen_permutahedron}
        \genPerm{U} =
        \sum_{i=0}^a \Delta_{\{j:\ (i,j)\in U\}}.
    \end{align}
    This polytope is a \emph{generalized permutahedron} in the sense of~\cite{postnikov_permutahedra_2009}, and we are interested in subdivisions of it whose cells are also Minkowski sums. 

    Given a $U$-tree $T$, we define the cell 
    \begin{align}
        \genPerm{T} = \sum_{i=0}^a \Delta_{\{j:\ (i,j)\in T\}}.
    \end{align}

    Applying the Cayley trick~\cite[Theorem~3.1]{huber_rambau_santos_cayleytrick_2000} to the triangulation in Theorem~\ref{thm_noncrossing_triangulation}, gives rise to the following \emph{coherent fine mixed subdivision} of the polytope $\genPerm{U}$. We refer to~\cite{huber_rambau_santos_cayleytrick_2000,santos_cayley_product_2005} for a detailed explanation on this terminology. 

    \begin{corollary}\label{cor_fine_mixed_subdivision}
        The polyhedral cells $\genPerm{T}$ as $T$ ranges over all $U$-trees are the maximal dimensional cells of a coherent fine mixed subdivision $\MixedSubdivision{U}$ of the polytope~$\genPerm{U}$. 
        Moreover, this subdivision satisfies the following properties: 
        \begin{enumerate}
            \item The face poset $\MixedSubdivision{U}$ is isomorphic to the restriction of the face poset of the $U$-Tamari complex $\TamComplex{U}$ to the faces that contain at least one element per column.\label{item_one_cor_fine_mixed_subdivision} 
            \item The poset of interior faces of $\MixedSubdivision{U}$ is isomorphic to the poset of interior faces of $\TamComplex{U}$. 
            \label{item_two_cor_fine_mixed_subdivision} 
            \item The dual graph of $\MixedSubdivision{U}$ is the Hasse diagram of the $U$-Tamari lattice~$\Tam{U}$.
        \end{enumerate}
    \end{corollary}

    \begin{proof}
        It suffices to prove Parts~\eqref{item_one_cor_fine_mixed_subdivision} and \eqref{item_two_cor_fine_mixed_subdivision}. They follow directly from~\cite[Theorem~3.1]{huber_rambau_santos_cayleytrick_2000}.   
    \end{proof}

    Although $\genPerm{U}\subseteq \mathbb R^{b+1}$, this polytope lies in the affine space determined by the sum of the coordinates equal to $a+1$. So, we can draw the subdivision in Corollary~\ref{cor_fine_mixed_subdivision} in $\mathbb R^b$. \Cref{fig_fineSubdivisionENEN,fig_fineSubdivisionNENEN,fig_fineSubdivisionsEENEEN} illustrate several examples.

    \begin{example}[The Tamari lattice]\label{ex_fineSubdivisonENEN}
        The classical Tamari lattice is the dual of the fine mixed subdivision~$\MixedSubdivision{U}$ of~$\genPerm{U}$, where $U$ is the set of lattice points weakly above the path $\nu=(EN)^n$.
        For $n=2$, we have $\nu=ENEN$  and the polytope~$\genPerm{U}=\Delta_{012}+\Delta_{012}+\Delta_{01}$ is illustrated in Figure~\ref{fig_fineSubdivisionENEN}.
    \end{example}

    \begin{figure}[htb]
        \centering
        \input{figures/fineSubdivisionENEN}
        \caption{The classical Tamari lattice as the dual of the fine mixed subdivision~$\MixedSubdivision{U}$ for the stack set $U$ of lattice points weakly above the path~$\nu=ENEN$.}
        \label{fig_fineSubdivisionENEN}
    \end{figure}

    This polytope is subdivided into five cells corresponding to the five $\nu$-trees, which are also shown in the figure. Their Minkowski sums are shown in Figure~\ref{fig_fineSubdivisionENEN_cells}.

    \begin{figure}[htb]
        \begin{tabular}{ccccc}

\nuTree{1,1,0}{1,1,0}{0.65}{}{}{}{
\node at (-1,3) {$j\backslash i$};
\node at (-1,2) {$0$};
\node at (-1,1) {$1$};
\node at (-1,0) {$2$};
\node at (0,3) {$0$};
\node at (1,3) {$1$};
\node at (2,3) {$2$};
\node at (1,-1) {$\Delta_{012}+\Delta_{2}+\Delta_{1}$};
}

&
\nuTree{1,1,0}{0,2,0}{0.65}{}{}{}{
\node at (1,-1) {$\Delta_{01}+\Delta_{12}+\Delta_{1}$};
}
&
\nuTree{1,1,0}{0,1,1}{0.65}{}{}{}{
\node at (1,-1) {$\Delta_{0}+\Delta_{012}+\Delta_{1}$};
}
&
\nuTree{1,1,0}{0,0,2}{0.65}{}{}{}{
\node at (1,-1) {$\Delta_{0}+\Delta_{02}+\Delta_{01}$};
}
&
\nuTree{1,1,0}{1,0,1}{0.65}{}{}{}{
\node at (1,-1) {$\Delta_{02}+\Delta_{2}+\Delta_{01}$};
}

\end{tabular}
        \caption{Minkowski sums of the cells in~\Cref{fig_fineSubdivisionENEN}.}
        \label{fig_fineSubdivisionENEN_cells}
    \end{figure}

    \begin{remark}[The Pitman--Stanley polytope]
        Alternatively, the classical Tamari lattice can be obtained as the dual of the fine mixed subdivision~$\MixedSubdivision{U}$ of~$\genPerm{U}$, where $U$ is the set of lattice points weakly above the path $\nu=(NE)^{n+1}$.\footnote{Note that adding an east step at the end of the path $\nu$ transforms the polytope $\genPerm{U}$ by translation by~$e_0$, while adding a north step at the beginning increases the dimension by one. In Example~\ref{ex_fineSubdivisonENEN}, we removed the initial north step and final east step because they are irrelevant for the definition of the lattice. Taking that convention, the dimension of the resulting polytope~$\genPerm{U}$ coincides with that of the corresponding associahedron.}
        In this case, $\genPerm{U}=\sum_{k=0}^{n+1} \Delta_{\{0,1,\dots,k\}}$ which is integrally equivalent to the Pitman--Stanley polytope $\Pi_{n+1}(x)$ from~\cite{PitmanStanleyPolytope}, for $x=1^{n+1}=(1,\dots,1)$. 
        
        The case $n=2$ is illustrated in~\Cref{fig_fineSubdivisionNENEN}. The path $\nu=NENENE$ and the polytope~$\genPerm{U}$ is the Minkowski sum~$\genPerm{U}=\Delta_{0123}+\Delta_{012}+\Delta_{01}+\Delta_{0}$, which is integrally equivalent to the Pitman--Stanley polytope~$\Pi_3(1,1,1)$. It is subdivided into five cells corresponding to the five $\nu$-trees (classical rooted binary trees).

        \begin{figure}[htb]
            \centering
            \input{figures/fineSubdivisionNENEN}
            \input{figures/fineSubdivisionNENEN_expanded}
            \caption{Alternative realization of the classical Tamari lattice as the dual of the fine mixed subdivision~$\MixedSubdivision{U}$ for the stack set $U$ of lattice points weakly above the path~$\nu=NENENE$.}
            \label{fig_fineSubdivisionNENEN}
        \end{figure}

        Recall that for $x=(x_0,x_1,\dots, x_n)$ with $x_i>0$ for all $i$, the \defn{Pitman--Stanley polytope} is defined as
        \[
        \Pi_{n+1}(x) = \left\{ 
        (y_0,y_1,\dots,y_n)\in \mathbb{R}^{n+1}:\ y_i\geq 0 \text{ and } y_0+\dots +y_i \leq x_0+\dots +x_i \text{ for } 0\leq i \leq n 
        \right\}.
        \] 
        If $U$ is the set of lattice points weakly above $\nu=(E^{m_0}N)(E^{m_1}N)\dots (E^{m_n}N)$, then 
        \[
        \genPerm{U} = 
        \Delta_{0,\dots,n+1} + 
        \sum_{i=0}^n m_i \Delta_{0, \dots,n+1-i}
        \] 
        which is integrally equivalent to $\Pi_{n+1}(m_0+1,m_1,m_2,\dots , m_n)$ via the map 
        \[
        (y_0,y_1,\dots , y_{n+1}) \to 
        (y_{n+1},\dots,y_1).
        \]
    \end{remark}

    \begin{example}[Some alt $\nu$-Tamari lattices]\label{ex_fineSubdivisionsEENEEN}
        Let $\nu=EENEEN=(2,2,0)$ and $\delta$ be an increment vector with respect to $\nu$, i.e. $\delta=(2,0)$, $(1,0)$ or $(0,0)$. The alt $\nu$-Tamari lattice~$\altTam{\nu}{\delta}$ is the dual of the fine mixed subdivision~$\MixedSubdivision{U_{\delta,\nu}}$ of $\genPerm{U_{\delta,\nu}}$. The subdivisions for the three possible values of $\delta$ are shown in~\Cref{fig_fineSubdivisionsEENEEN}.
    \end{example}
    
    \begin{figure}[htb]
        \centering
        \input{figures/fineSubdivisionsEENEEN}
        \caption{Examples of the fine mixed subdivisions~$\MixedSubdivision{U_{\delta,\nu}}$ of $\genPerm{U_{\delta,\nu}}$ for the path $\nu=EENEEN=(2,2,0)$ and the three possible choices of $\delta$: $(2,0)$, $(1,0)$, and $(0,0)$.}
        \label{fig_fineSubdivisionsEENEEN}
    \end{figure}


    As we can see in Figure~\ref{fig_fineSubdivisionsEENEEN}, the shown fine mixed subdivisions are different, but all of them subdivide the same polytope. This follows from the following fact. 

    \begin{proposition}
        Let $\delta$ be an increment vector with respect to a lattice path $\nu$, and~$U=U_{\delta,\nu}$ be its corresponding shape. Then
        $
        \genPerm{U}=  \genPerm{\nu},
        $
        where 
        \[
        \genPerm{\nu}=
                \sum_{i=0}^a \Delta_{\{j:\ (i,j) \text{ is weakly above } \nu\}}.
        \]
        Therefore, for a fixed $\nu$, the Hasse diagram of every alt $\nu$-Tamari lattice $\altTam{\nu}{\delta}$ can be realized as the dual graph of fine mixed subdivision of the same polytope $\genPerm{\nu}$. 
    \end{proposition}
    \begin{proof}
        Note that the polytope $\genPerm{\nu}=\genPerm{U_{\delta^{\max},\nu}}$, where $\delta^{\max}=(\nu_1,\dots,\nu_b)$. 
        By Remark~\ref{rem_shape_permutation_columns}, the two sets $U_\delta$ and $U_{\delta^{\max}}$ differ by a permutation of columns. This permutation produces a reordering of the summands in Equation~\eqref{eq_gen_permutahedron}. Since such a reordering does not alter the result of the Minkowski sum, 
        the result follows. 
    \end{proof}

    \begin{remark}
        We can also give an alternative definition of the polytope $\genPerm{U}$ and the cells $\genPerm{T}$ by taking Minkowski sums of rows instead of columns: 
        \[
        \genPerm{U}' = \sum_{j=0}^b \Delta_{\{i:\ (i,j)\in U\}}, 
        \qquad 
        \genPerm{T}' = \sum_{j=0}^b \Delta_{\{i:\ (i,j)\in T\}}.
        \]
        In this set up, we would get a realization of the $U$-Tamari lattice as the dual of a fine mixed subdivision of a polytope in an affine space of dimension $a$ in $\mathbb R^{a+1}$ ($\cong \mathbb R^a$). 
    \end{remark}

\subsection{Tropical hyperplane arrangement realization}\label{sec_tropical}

    In this section we present our third geometric realization of the $U$-Tamari lattice using the duality between regular triangulations of (a subpolytope of) $\Delta_a\times \Delta_b$ and tropical hyperplane arrangements conceived in~\cite{develin_tropical_2004} and further developed in~\cite{ardila_tropical_2009,fink_stiefel_2015}.
    
    Intuitively, the cell complex induced by the tropical hyperplane arrangement is dual to the fine mixed subdivision $\MixedSubdivision{U}$ in Corollary~\ref{cor_fine_mixed_subdivision}, as illustrated in the example in~\Cref{fig_fineSubivisionENEN_tropical} (c.f.~\Cref{fig_fineSubdivisionENEN}).

    \begin{figure}[htb]
        \centering
        \input{figures/fineSubdivisionENEN_tropical}
        \caption{An example of the duality between fine mixed subdivisions and arrangements of (degenerate) tropical hyperplanes. }
        \label{fig_fineSubivisionENEN_tropical}
    \end{figure}

    Throughout the section, we fix a non-crossing height function $h$ with respect to a unimodular shape~$U$, and assume that the smallest rectangle containing $U$ is an~$a\times b$ rectangle. We set $h(i,j)=-\infty$, for every position $(i,j)$ inside the $a\times b$ rectangle that is not inside $U$. 

    The \defn{tropical projective space} is the space
    \[
    \mathbb{TP}^b =
    \left(
    (\mathbb R \cup -\infty )^{b+1} \setminus (-\infty,\dots ,-\infty)/\mathbb{R}(1,\dots ,1)
    \right).    
    \]
    The \defn{tropical semiring} is the tuple $(\mathbb R\cup -\infty,\oplus, \otimes)$, where the tropical addition $\oplus$ and the tropical multiplication $\otimes$ are defined by $m\oplus n = \max \{m,n\}$ and $m\otimes n=m+n$.
    
    We consider the arrangement $\arrangement=(H_i)_{0\leq i\leq a}$ of \defn{inverted tropical hyperplanes} centered at the points~$v_i=(h(i,j))_{0\leq j\leq b}$, which are explicitly defined by
    \begin{align}\label{eq_tropical_hyperplanes}
        H_i = 
        \left\{
        y \in \mathbb{TP}^b : \min_{(i,j)\in U} \{ -h(i,j)+y_j \} \text{ is attained twice} 
        \right\}.
    \end{align}
    In the case where some of the $h(i,j)$ are equal to $-\infty$, the corresponding $H_i$ is a \defn{degenerate tropical hyperplane}, see~\cite{fink_stiefel_2015}. 
    The third and fourth illustrations in Figure~\ref{fig_fineSubivisionENEN_tropical} show three tropical hyperplanes in~$\mathbb{TP}^2\cong \mathbb{R}^2$, one of which (in black) is degenerate.  
    They are drawn using two different identifications of~$\mathbb{TP}^2$ in~$\mathbb{R}^3$; the first is obtained by intersecting with the plane $y_0+y_1+y_2=0$, while the second with the plane $y_0=0$. 
    
    Figure~\ref{fig_tropical_region_labels} illustrates a single tropical hyperplane $H_i$ intersected with the plane~$y_0=h(i,0)$. If all~$h(i,j)$ are finite, then $H_i$ subdivides the space $\mathbb{TP}^d$ in $d+1$ regions, where the minimum in Equation~\eqref{eq_tropical_hyperplanes} is attained only once. In Figure~\ref{fig_tropical_region_labels} (left) we are considering the case $d=2$, and the three regions are labeled by the term that achieves the minimum. The \defn{center} $v_i\in\mathbb{TP}^d$ of the tropical hyperplane $H_i$ is the point where the all the terms in Equation~\eqref{eq_tropical_hyperplanes} are equal; For instance, for $d=2$ we have $v_i=(h(i,0),h(i,1),h(i,2))$. In some of the $h(i,j)=-\infty$, then we get a degenerate tropical hyperplane; Figure~\ref{fig_tropical_region_labels} (right) shows a degenerate tropical hyperplane $\min\{-h(i,0)+y_0,-h(i,1)+y_1\}$ when $d=2$ and $h(i,2)=-\infty$. Note that setting $h(i,j)=-\infty$ is equivalent to removing the term~$-h(i,j)+y_j$ from Equation~\eqref{eq_tropical_hyperplanes}, because the minimum is never achieved at this term in that case.

    \begin{figure}[htb]
        \centering
        \input{figures/tropical_region_labels}
        \caption{The tropical hyperplane $H_i$ in $\mathbb{TP}^2$ drawn at the intersection with the plane $y_0=h(i,0)$. The regions are labeled by the term of Equation~\eqref{eq_tropical_hyperplanes} where the minimum is attained. The right is a degenerate tropical hyperplane, assuming \mbox{$h(i,2)=-\infty$}.}
        \label{fig_tropical_region_labels}
    \end{figure}

    The following result was proven by Develin and Sturmfels~\cite{develin_tropical_2004} in the case of subdivisions of products of two simplices $\Delta_a\times \Delta_b$, and generalized to subdivisions of a subpolytope $\subpolytope{U}\subseteq \Delta_a\times \Delta_b$ by Fink and Rinc\'on in~\cite{fink_stiefel_2015}. The last amounts to setting the heights of points outside $U$ to be equal $-\infty$.     

    \begin{theorem}[{\cite[Proof of Theorem~1]{develin_tropical_2004}, \cite[Section~4]{fink_stiefel_2015}}]
    \label{thm_develin_sturmfels}
        The arrangement~$\arrangement$ induces a polyhedral decomposition of $\mathbb{TP}^b$, whose poset of bounded faces is anti-isomorphic to the poset of interior faces of the subdivision $\subpolytope{U}^h$ of $\subpolytope{U}$ induced by $h$.
    \end{theorem}

    \begin{figure}[htb]
        \centering
        \input{figures/tropicalArrangementsEENEEN}
        \caption{Three arrangements $\arrangement$ of tropical hyperplanes dual to the fine mixed subdivisions in~\Cref{fig_fineSubdivisionsEENEEN}.}
        \label{fig_tropicalArrangementsEENEEN}
    \end{figure}

    \begin{example}[Example~\ref{ex_fineSubdivisionsEENEEN} continued]
    Figure~\ref{fig_tropicalArrangementsEENEEN} illustrates three arrangements $\arrangement$ of tropical hyperplanes (drawn at the plane $y_0=0$) induced by the height function $h:U\to \mathbb{R}$, defined by $h(i,j)=ij$ on each of the three stack sets $U$ shown (and equal to $-\infty$ outside $U$).

    By~\Cref{thm_noncrossing_triangulation} and~\Cref{ex_canonical_height_function}, $\subpolytope{U}^h$ is a triangulation of $\subpolytope{U}$ which is equal to the $U$-Tamari complex~$\TamComplex{U}$. 
    Applying the Cayley trick (see \Cref{cor_fine_mixed_subdivision}), the poset of interior faces of $\TamComplex{U}$ is isomorphic to the poset of interior faces of the fine mixed subdivision $\MixedSubdivision{U}$ of $\genPerm{U}$. 
    The three fine mixed subdivisions for the different $U$'s shown are illustrated in~\Cref{fig_fineSubdivisionsEENEEN}. 

    Comparing both figures, one can see that the poset of interior faces of the three fine mixed subdivisions in~\Cref{fig_fineSubdivisionsEENEEN} is anti-isomorphic to the poset of bounded faces of the corresponding arrangements of tropical hyperplanes in~\Cref{fig_tropicalArrangementsEENEEN}. 
    In particular, the graph of bounded edges of the arrangement~$\arrangement$ is a geometric realization of the $U$-Tamari lattice. 
    \end{example}

\section{The $U$-associahedron}\label{sec_Uassociahedron}

In general, we can combine~\Cref{thm_develin_sturmfels} and~\Cref{thm_noncrossing_triangulation} to obtain a tropical geometric realization of the~$U$-Tamari lattice which we now make more explicit. 

    \begin{definition}\label{def_U_associhedron}
        Let $U$ be a stack set and $\arrangement$ be the arrangement of inverted tropical hyperplanes described in Equation~\eqref{eq_tropical_hyperplanes} associated to a non-crossing height function $h:U\rightarrow \mathbb{R}$.
        The \defn{$U$-associahedron} $\Asso{U}(h)$ is the polyhedral complex of bounded cells induced by $\arrangement$. 
        
        For an increment vector $\delta$ with respect to a lattice path $\nu$, we define the \defn{alt $\nu$-associahedron} $\altAsso{\delta}{\nu}(h)$ as~$\Asso{U_{\delta,\nu}}(h)$. 
        To simplify notation, sometimes we omit~$h$ when it is clear from the context.  
    \end{definition}

    
    The following is an alternative definition of the $U$-associahedron in purely combinatorial terms.
    
    \begin{theorem}[{Cf.~\cite[Theorem~5.2]{ceballos_geometry_2019}}]
    \label{def_thm_U_associahedron}
        For a non-crossing height function $h$, the $U$-associahedron $\Asso{U}(h)$ is a polyhedral complex whose poset of cells is anti-isomorphic to the poset of interior faces of the $U$-Tamari complex. In particular:
        \begin{enumerate}
            \item Its vertices correspond to $U$-trees.
            \item Two vertices are connected by an edge if and only if the corresponding $U$-trees are related by rotation. 
            That is, the edge graph of $\Asso{U}(h)$ is the Hasse diagram of the $U$-Tamari lattice. 
        \end{enumerate}
    \end{theorem}
    \begin{proof}
        This is a direct consequence of~\Cref{thm_noncrossing_triangulation} via the ``tropicalization" in~\Cref{thm_develin_sturmfels}.
    \end{proof}

\subsection{Defining inequalities}

    In the following two sections, we concentrate on more explicit descriptions of the $U$-associahedron. We provide explicit defining inequalities and the coordinates of its vertices in terms of the height function $h$. This is obtained by explicitly writing down the inequalities induced by the tropical hyperplanes in Equation~\eqref{eq_tropical_hyperplanes}. 
    
    \begin{definition}
        For $(i,j)\in U$ we define the polyhedron $\widetilde g(i,j)$ in $\mathbb R^{b+1}$ by
        \begin{align}
        \widetilde g(i,j) :=         \left\{
        y \in \mathbb R^{b+1} : 
        -h(i,j)+y_j \leq 
        -h(i,j')+y_{j'} 
        \text{ for } (i,j')\in U 
        \right\}.            
        \end{align}        
        This is the closure of the connected component of $\mathbb R^{b+1}\setminus H_i$ where the minimum of Equation~\eqref{eq_tropical_hyperplanes} is attained at $-h(i,j)+y_j$. 

        For a subset $A$ of $U$, define the polyhedron $\widetilde g(A)$ in $\mathbb R^{b+1}$ as
        \begin{align}
        \widetilde g(A) := \bigcap_{(i,j)\in A} \widetilde g(i,j).
        \end{align}
        If $\widetilde g(A)$ is non-empty then it contains the vector $(1,\dots, 1)$ in its lineality space. Intersecting with the hyperplane $y_0=0$, we obtain a polyhedral cell $g(A)$ in $\mathbb R^b$:
        \begin{align}
        g(A) := \widetilde g(A) \cap \{y_0=0\} = \bigcap_{(i,j)\in A} g(i,j), 
        \end{align}
        where $g(i,j):= \widetilde g(i,j) \cap \{y_0=0\}$.
    \end{definition}    

    The following two lemmas and corollary follow from the duality between the triangulation $\subpolytope{U}^h$ of~$\subpolytope{U}$ induced by a non-crossing height function $h$ and the arrangement of tropical hyperplanes~$\arrangement$~\cite{develin_tropical_2004,fink_stiefel_2015}.

    \begin{lemma}[{Cf. \cite[Lemma 5.4]{ceballos_geometry_2019}}]
        The polyhedron $\widetilde g(A)$ is non-empty if and only if $A$ is a pairwise $U$-compatible set, i.e. a face of the $U$-Tamari complex $\TamComplex{U}$.
    \end{lemma}

    \begin{lemma}[{Cf. \cite[Lemma 5.5]{ceballos_geometry_2019}}]\label{lem_defining_inequalities}
        For each $U$-covering face $A$, $g(A)$ is a (bounded) convex polytope of dimension $a+b+1-|A|$, whose vertices correspond to the $U$-trees containing $A$:
        \[
        g(A) = \conv \left\{
        g(T): \ T \text{ is a $U$-tree containing } A  
        \right\}.                    
        \]
    \end{lemma}
    \begin{corollary}[{Cf. \cite{ceballos_geometry_2019}}]\label{cor_associahedron_faces}
        The $U$-associahedron is the polyhedral complex
        \[
            \Asso{U}(h) = \left\{
                g(A):\, A \text{ is a } U\text{-covering face}
            \right\}.
        \]
    \end{corollary}

\subsection{Coordinates of the vertices}

    The defining inequalities of the faces of the $U$-associahedron in~\Cref{lem_defining_inequalities} and~\Cref{cor_associahedron_faces} can be used to give a precise formula for the coordinates of its vertices. For this, the following lemma is useful. 

    \begin{lemma}\label{lem_vertexcoordinates_one}
        Let $T$ be a $U$-tree and $g(T)=(y_1,\dots,y_b)$ be the coordinate of its corresponding vertex of the $U$-associahedron~$\Asso{U}(h)$. 
        If $(i,j),(i,j')\in T$ for some $i$ and $j\neq j'$ then 
        \begin{align}
            y_{j}-y_{j'}=h(i,j)-h(i,j').
        \end{align}
    \end{lemma}
    \begin{proof}
        Since $(y_1,\dots,y_b)\in g(i,j) \cap g(i,j')$ then 
        \begin{align*}
            -h(i,j)+y_j \leq   -h(i,j')+y_{j'} \\ 
            -h(i,j')+y_{j'} \leq   -h(i,j)+y_{j}. 
        \end{align*}
        From this we deduce that both inequalities are equalities, and so $y_{j}-y_{j'}=h(i,j)-h(i,j')$.
    \end{proof}

    Recall that a point $(i,j)\in U$ is the point located at column $i$ and row $j$, where the row labels $0,\dots,a$ increase from left to right and the column labels $0,\dots,b$ increase from top to bottom. 
    
    Let $T$ be a $U$-tree and $1\leq k \leq b$.
    We denote by $p_k(T)\in T$ the left most node of $T$ in row $k$, 
    and by~$p_k'(T)\in T$ the right most node in row 0 in the unique path from $p_k(T)$ to the root in $T$. 
    As illustrated in~\Cref{fig_path_to_root},
    the unique path from $p_k(T)$ to $p_k'(T)$ in $T$ (shown in color blue) consists of a sequence of up vertical runs and left horizontal runs that alternate. 
    We denote by $P_k^+(T)$ the \defn{set of bottom nodes}, and by $P_k^-(T)$ the \defn{set of top nodes}, of the vertical runs; these are labeled $+$ and $-$, respectively, in~\Cref{fig_path_to_root}. 
    And we let $P_k(T)=P_k^+(T)\sqcup P_k^-(T)$.    

    \begin{figure}[htb]
        \centering
        \input{figures/path_to_root}
        \caption{The set $P_k^+(T)$ is the set of nodes labeled $+$ in the figure, while~$P_k^-(T)$ is the set of nodes labeled $-$. The blue path is the unique path from the left most node of $T$ in row $k$ to the top of the shape.}
        \label{fig_path_to_root}
    \end{figure}

    \begin{proposition}[{Cf. \cite[Lemma 5.6]{ceballos_geometry_2019}}]\label{prop_vertexcoordinates}
        Let $T$ be a $U$-tree and $g(T)=(y_1,\dots,y_b)$ be the coordinate of its corresponding vertex of the $U$-associahedron~$\Asso{U}(h)$. Then 
        \begin{align}\label{eq_vertex_as_alternating_sum}
        y_k= \sum_{(i,j)\in P_k(T)} \pm h(i,j),            
        \end{align}
        where the sign is positive if $(i,j)\in P_k^+(T)$ and negative if $(i,j)\in P_k^-(T)$.
    \end{proposition}

    \begin{proof}
        Let $\ell$ be the number of up vertical runs of the path from $p_k(T)$ to $p_k'(T)$ in $T$. These vertical runs alternate with $\ell-1$ left horizontal runs, see~\Cref{fig_path_to_root}. 

        Let $(i_1,j_1),\dots , (i_\ell,j_\ell)\in P_k^+(T)$ be the elements of $P_k^+(T)$ ordered in ``decreasing order'' 
        \begin{align*}
         k=&j_1\geq j_2 \geq \dots \geq j_\ell >0\\
         &i_1\geq i_2\geq \dots \geq i_\ell
        \end{align*}
        Therefore, $P_k^-(T)=\{(i_1,j_2),(i_2,j_3)\dots , (i_\ell,j_{\ell+1})\}$ where $j_{\ell+1}=0$. 

        By~\Cref{lem_vertexcoordinates_one}, we have that
        \begin{align*}
            y_{j_1} - y_{j_2} &= h(i_1,j_1) - h(i_1,j_2) \\
            y_{j_2} - y_{j_3} &= h(i_2,j_2) - h(i_2,j_3) \\
            &\vdots \\
            y_{j_\ell} - y_{j_{\ell+1}} &= h(i_\ell,j_\ell) - h(i_\ell,j_{\ell+1}) 
        \end{align*}
        Adding up, we get 
        \[
        y_k=y_k-y_0= \sum_{(i,j)\in P_k(T)} \pm h(i,j)
        \]
        where the sign is positive if $(i,j)\in P_k^+(T)$ and negative if $(i,j)\in P_k^-(T)$, as desired. 
    \end{proof}

\section{A canonical realization}\label{sec_cononial_realization}

    We now have all the ingredients to provide a beautifully simple realization of the \mbox{$U$-associahedron} for any stack set $U$. 
    If $U$ is left aligned, this gives an explicit canonical realization of a $\nu$-associahedron~\cite{ceballos_geometry_2019}, which was not explicitly described before. 
    For a general $U$, our construction gives a geometric realization of an alt $\nu$-associehedron, whose edge graph is the Hasse diagram of the alt~$\nu$-Tamari lattice introduced by Ceballos and Chenevière in~\cite{ceballos_cheneviere_linear_2024}. This is the first known geometric realization of the alt $\nu$-Tamari lattice as the edge graph of a polytopal complex.
    
    Let $U$ be a stack set.  The \defn{canonical height function} is the map defined by
    \begin{align}
        h:U&\rightarrow \mathbb{R} \\
        h(i,j)&=ij
    \end{align}
    It is easy to check that this is a non-crossing height function in the sense of~\Cref{def_height_function}. Therefore, it induces a geometric realization of the $U$-associahedron~$\Asso{U}(h)$, see~\Cref{def_thm_U_associahedron}.
    It turns out that for this specific choice of height function, the coordinates of the vertices become elegant and simple. 

    \begin{figure}[htb]
        \centering
        \input{figures/path_to_root_area}
        \caption{Area below the path $R_k(T)$.}
        \label{fig_path_to_root_area}
    \end{figure}

    \begin{theorem}\label{thm_canonical_realization}
        Let $U$ be a stack set and $h:U\rightarrow \mathbb{R}$ be the canonical height function $h(i,j)=ij$. The coordinate $g(T)=(y_1,\dots,y_k)$ of a $U$-tree $T$ in the $U$-associahedron~$\Asso{U}(h)$ is determined by
        \[
        y_k(T)=\area(R_k(T)),
        \]
        where $R_k(T)$ is the path that connects the left most node of $T$ in row $k$ to the root, and the area of a lattice path $R$ is the number of boxes below it in the smallest rectangle containing it. 
    \end{theorem}

    \begin{proof}
        The proof is a direct consequence of~\Cref{prop_vertexcoordinates} and is illustrated in~\Cref{fig_path_to_root_area}. 
        Rewritting Equation~\eqref{eq_vertex_as_alternating_sum} for the canonical height function we get 
        \[
        y_k = (i_1j_1-i_1j_2)+(i_2j_2-i_2j_3)+\dots+(i_\ell j_\ell-0).
        \]
        This is precisely, the area below the path $R_k(T)$, which is the colored region in~\Cref{fig_path_to_root_area}.
    \end{proof}

    We call the realization of the $U$-associahedron in~\Cref{thm_canonical_realization} the \defn{canonical realization}.

\begin{example}[The coordinate of a $\nu$-tree]
    Let $\nu=(2,1,2,1,0)$, and $\mu=(0,2,0,3,1)$ be a $\nu$-path. The corresponding $\nu$-tree $T$ is shown in Figure~\ref{fig_vtree_coordinate}.
    The areas below the paths $R_1(T),R_2(T),R_3(T)$ and $R_4(T)$ are colored red, blue, green and purple respectively. Counting the number of boxes in these areas we get the coordinates of the vector $g(T)=(1,6,3,5)$. 
\end{example}

\begin{figure}[h]
    \centering
    \begin{tabular}{cccc}
\nuTree{2,1,2,1,0}{0,2,0,3,1}{.5}{}{0/3/$1$/red/east}{}{
\filldraw[nearly transparent, red] (1,3) -- (1,4) -- (0,4) -- (0,3);
}
&
\nuTree{2,1,2,1,0}{0,2,0,3,1}{.5}{}{0/2/$6$/blue/east}{}{
\filldraw[nearly transparent, blue] (5,2) -- (5,3) -- (1,3) -- (1,4) -- (0,4) -- (0,2);
}
&
\nuTree{2,1,2,1,0}{0,2,0,3,1}{.5}{}{0/1/$3$/green/east}{}{
\filldraw[nearly transparent, green] (1,1) -- (1,4) -- (0,4) -- (0,1);
}
&
\nuTree{2,1,2,1,0}{0,2,0,3,1}{.5}{}{0/0/$5$/purple/east}{}{
\filldraw[nearly transparent, purple] (2,0) -- (2,1) -- (1,1) -- (1,4) -- (0,4) -- (0,0);
}
\end{tabular}
    \caption{The coordinate of the shown $\nu$-tree $T$ is $g(T)=(y_1,y_2,y_3,y_4)=(1,6,3,5)$.}
    \label{fig_vtree_coordinate}
\end{figure}

\begin{example}[The associahedron]
The classical $n$-dimensional associahedron is obtained for \mbox{$\nu=(EN)^n$}.
Our canonical realization for $n=2$ is shown in Figure~\ref{fig_associahedron2D_paths_trees}.
The computation of the coordinates for the five trees is shown in Figure~\ref{fig_vtree_coordinate_asso2}. The coordinates are:
\[
\begin{array}{ccccc}
(0,0),
&
(0,1),
&
(1,2),
&
(2,2),
&
(2,0).
\end{array}
\]
\begin{figure}[h]
    \centering
    \begin{tabular}{ccccc}
\nuTree{1,1,0}{1,1,0}{0.7}{}{-0.2/1/$0$/red/east,-0.2/0/$0$/blue/east}{}{} 
&
\nuTree{1,1,0}{0,2,0}{0.7}{}{-0.2/1/$0$/red/east,-0.2/0/$1$/blue/east}{}{
\filldraw[nearly transparent, blue] (1,0) -- (1,1) -- (0,1) -- (0,0);
}
&
\nuTree{1,1,0}{0,1,1}{0.7}{}{-0.2/1/$1$/red/east,-0.2/0/$2$/blue/east}{}{
\filldraw[nearly transparent, red] (1,1) -- (1,2) -- (0,2) -- (0,1);
\filldraw[nearly transparent, blue] (1,0) -- (1,2) -- (0,2) -- (0,0);
}
&
\nuTree{1,1,0}{0,0,2}{0.7}{}{-0.2/1/$2$/red/east,-0.2/0/$2$/blue/east}{}{
\filldraw[nearly transparent, red] (2,1) -- (2,2) -- (0,2) -- (0,1);
\filldraw[nearly transparent, blue] (1,0) -- (1,2) -- (0,2) -- (0,0);
}
&
\nuTree{1,1,0}{1,0,1}{0.7}{}{-0.2/1/$2$/red/east,-0.2/0/$0$/blue/east}{}{
\filldraw[nearly transparent, red] (2,1) -- (2,2) -- (0,2) -- (0,1);
}
\end{tabular}
    \caption{The coordinates of the vertices of the canonical realization of the \mbox{2-dimensional} associahedron are: $(0,0),(0,1),(1,2),(2,2),(2,0)$.}
    \label{fig_vtree_coordinate_asso2}
\end{figure}

\begin{figure}[htb]
    \centering
    \input{figures/associahedronENEN_paths_trees}
    \caption{The 2-dimensional associahedron.}
    \label{fig_associahedron2D_paths_trees}
\end{figure}

Our canonical realization for $n=3$ is shown in Figure~\ref{associahedronENENEN_paths_trees}. 
The computation of the coordinates for the 14 trees for $n=3$ is shown in Figure~\ref{fig_vtree_coordinate_asso3}. The coordinates are:
\[
\begin{array}{ccccccc}
(0,0,0)&
(0,0,1)&
(0,1,2)&
(0,2,2)&
(1,3,3)&
(2,4,3)&
(3,4,3)
\\
(0,2,0)&
(2,4,0)&
(3,4,0)&
(3,0,0)&
(3,0,1)&
(3,2,3)&
(1,2,3)
\end{array}
\]
\begin{figure}[h]
    \centering
    \begin{tabular}{ccccccc}
\nuTree{1,1,1,0}{1,1,1,0}{0.4}{}{-0.2/2/$0$/red/east,-0.2/1/$0$/blue/east,-0.2/0/$0$/green/east}{}{}
&
\nuTree{1,1,1,0}{0,2,1,0}{0.4}{}{-0.2/2/$0$/red/east,-0.2/1/$0$/blue/east,-0.2/0/$1$/green/east}{}{
\filldraw[nearly transparent, green] (1,0) -- (1,1) -- (0,1) -- (0,0);
}
&
\nuTree{1,1,1,0}{0,1,2,0}{0.4}{}{-0.2/2/$0$/red/east,-0.2/1/$1$/blue/east,-0.2/0/$2$/green/east}{}{
\filldraw[nearly transparent, blue] (1,1) -- (1,2) -- (0,2) -- (0,1);
\filldraw[nearly transparent, green] (1,0) -- (1,2) -- (0,2) -- (0,0);
}
&
\nuTree{1,1,1,0}{0,0,3,0}{0.4}{}{-0.2/2/$0$/red/east,-0.2/1/$2$/blue/east,-0.2/0/$2$/green/east}{}{
\filldraw[nearly transparent, blue] (2,1) -- (2,2) -- (0,2) -- (0,1);
\filldraw[nearly transparent, green] (1,0) -- (1,2) -- (0,2) -- (0,0);
}
&
\nuTree{1,1,1,0}{0,0,2,1}{0.4}{}{-0.2/2/$1$/red/east,-0.2/1/$3$/blue/east,-0.2/0/$3$/green/east}{}{
\filldraw[nearly transparent, red] (1,2) -- (1,3) -- (0,3) -- (0,2);
\filldraw[nearly transparent, blue] (2,1) -- (2,2) -- (1,2) -- (1,3) -- (0,3) -- (0,1);
\filldraw[nearly transparent, green] (1,0) -- (1,3) -- (0,3) -- (0,0);
}
&
\nuTree{1,1,1,0}{0,0,1,2}{0.4}{}{-0.2/2/$2$/red/east,-0.2/1/$4$/blue/east,-0.2/0/$3$/green/east}{}{
\filldraw[nearly transparent, red] (2,2) -- (2,3) -- (0,3) -- (0,2);
\filldraw[nearly transparent, blue] (2,1) -- (2,3) -- (0,3) -- (0,1);
\filldraw[nearly transparent, green] (1,0) -- (1,3) -- (0,3) -- (0,0);
}
&
\nuTree{1,1,1,0}{0,0,0,3}{0.4}{}{-0.2/2/$3$/red/east,-0.2/1/$4$/blue/east,-0.2/0/$3$/green/east}{}{
\filldraw[nearly transparent, red] (3,2) -- (3,3) -- (0,3) -- (0,2);
\filldraw[nearly transparent, blue] (2,1) -- (2,3) -- (0,3) -- (0,1);
\filldraw[nearly transparent, green] (1,0) -- (1,3) -- (0,3) -- (0,0);
}
\\
\end{tabular}
    \caption{The coordinates of the vertices of the canonical realization of the \mbox{3-dimensional} associahedron.}
    \label{fig_vtree_coordinate_asso3}
\end{figure}

\end{example}

\begin{figure}[htb]
    \centering
    \input{figures/associahedronENENEN_paths_trees}
    \caption{The 3-dimensional associahedron.}
    \label{associahedronENENEN_paths_trees}
\end{figure}

\begin{example}[The Fuss-Catalan associahedron]
    For $m,n\in \mathbb{N}$, the $n$-dimensional $m$-associahedron is obtained for $\nu=(E^mN)^n$. Our canonical realizations for $(m,n)=(2,2)$ and $(m,n)=(3,2)$ are shown in Figure~\ref{fig_Fuss_associahedron2D_m2} and Figure~\ref{fig_Fuss_associahedron2D_m3}, respectively.
    The coordinates of the vertices are special cases Formula~\eqref{eq_coord2D_nuAsso} in the more general example below. 
\end{example}

\begin{figure}[htb]
    \centering
    \input{figures/associahedronEENEEN_paths_trees}
    \caption{The Fuss-Catalan associahedron for $(m,n)=(2,2)$.}
    \label{fig_Fuss_associahedron2D_m2}
\end{figure}

\begin{figure}[htb]
    \centering
    \input{figures/associahedronEEENEEEN_paths_trees}
    \caption{The Fuss-Catalan associahedron for $(m,n)=(3,2)$.}
    \label{fig_Fuss_associahedron2D_m3}
\end{figure}

\begin{example}[Two dimensional $\nu$-associahedra] 
For $a,b\in \mathbb{N}$ and $\nu=E^aNE^bN$, a schematic illustration of the $\nu$-associahedron is shown in Figure~\ref{fig_associahedron_nu_2D}.
For this case, the $\nu$-paths are of the form $\mu=(r,s,a+b-r-s)$ where $r\leq a$ and $r+s\leq a+b$. 
The coordinates of the $\nu$-tree $T$ corresponding to $\mu$ are explicitly given by 
\begin{equation}\label{eq_coord2D_nuAsso}
    g(T) = 
\begin{cases}
    (a-r+b-s,2a-2r+b-s), & \text{for }  s \geq b \\
    (a+b-s,2a-2r), & \text{for }  s < b 
  \end{cases}
\end{equation}

In the first case, the parent of the left most node at height 0 is located at height 1, while in the second case the parent is located at height 2. Computing the areas below the paths $R_1(T)$ and $R_2(T)$ gives the desired formula.
\end{example}

\begin{figure}[htb]
    \centering
    \input{figures/associahedron_nu_2D}
    \caption{Schematic figure of the 2-dimensional $\nu$-associahedron for $\nu=E^aNE^bN$. In this case, $a=4$ and $b=5$.}
    \label{fig_associahedron_nu_2D}
\end{figure}

\begin{example}[Alt $\nu$-associahedra for $\nu=ENEEN$]
    For $\nu=ENEEN=(1,2,0)$, there are three possible choices for the increment vector $\delta$: $(2,0),(1,0)$ and $(0,0)$. 
    The canonical realization of the corresponding alt $\nu$-associahedra~$\altAsso{\delta}{\nu}$ are illustrated in Figure~\ref{fig_alt_associahedronENEEN_delta012}.
\end{example}    

\begin{figure}[p]
    \centering
    \input{figures/associahedronENEEN_delta2}
    \bigskip
    
    \input{figures/associahedronENEEN_delta1}
    \bigskip
    
    \input{figures/associahedronENEEN_delta0}
    \caption{Alt $\nu$ associahedra for $\nu=(1,2,0)$ and three different choices of increment vector $\delta$. The top is $\delta=(2,0)$, the middle is $\delta=(1,0)$, and the bottom is $\delta=(0,0)$.}
    \label{fig_alt_associahedronENEEN_delta012}
\end{figure}

\begin{example}[Alt $\nu$-associahedra for $\nu=EENEEN$]
    For $\nu=EENEEN=(2,2,0)$, there are three possible choices for the increment vector $\delta$: $(2,0),(1,0)$ and $(0,0)$.
    The case $\delta=(2,0)$ was illustrated in Figure~\ref{fig_Fuss_associahedron2D_m2}. The remaining two cases, $\delta=(1,0)$ and $\delta=(0,0)$, are shown in Figure~\ref{fig_alt_associahedronEENEEN_delta012}.
\end{example}    

\begin{figure}[p]
    \centering
    \input{figures/associahedronEENEEN_delta1}
    \bigskip
    
    \input{figures/associahedronEENEEN_delta0}
    \caption{Alt $\nu$ associahedra for $\nu=(2,2,0)$ and two different choices of increment vector $\delta$. The top is $\delta=(1,0)$ and the bottom is $\delta=(0,0)$. The case $\delta=(2,0)$ is the Fuss-Catalan associahedron in Figure~\ref{fig_Fuss_associahedron2D_m2}.}
    \label{fig_alt_associahedronEENEEN_delta012}
\end{figure}

\begin{example}[Some 3-dimensional alt $\nu$-associahedra]
Some examples of 3-dimensional alt $\nu$-associahedra are illustrated in~\Cref{fig_nuTamari_nDyck_ENEENN,fig_alt_nu_tamari_ENEENN_delta_one}.
Further examples are presented in~\Cref{sec_3D_examples}.
\end{example}




\section{From Loday to our canonical realization}\label{sec_Loday_to_canonical}
In this section, we show that Loday's associahedron is equivalent to our canonical realization of the associahedron via a simple affine transformation. 

As described in the introduction,~\Cref{sec_intro}, the coordinates of Loday's $n$-dimensional associahedron can be elegantly described as follows. Let $T$ be a plane binary tree with $n+1$ internal nodes. 
Loday's coordinate $L(T)=(\ell_1,\dots,\ell_{n+1})$ associated to $T$ satisfies 
$
\ell_i = a_ib_i,
$
where $a_i$ is the number of left descendant leaves of the $i$th internal node of $T$ in in-order, and $b_i$ is the number of its right descendant leaves. 
The coordinates of Loday's 2-dimensional associahedron are computed in~\Cref{fig_coordinates_Loday_2D}, and the polytope is illustrated on the left of~\Cref{fig_associahedron_Loday_canonical_2D}. 

\begin{figure}[htb]
    \centering
    \input{figures/coordinates_Loday_2D}
    \caption{Coordinates of the vertices of Loday's 2-dimensional associahedron.}
    \label{fig_coordinates_Loday_2D}
\end{figure}

For our canonical realization we draw the plane binary trees $T$ with $n+1$ internal nodes using lattice points in a staircase of size $n+1$ (i.e. $T$ is a $\nu$-tree for $\nu=(NE)^{n+1}$). The canonical coordinate $C(T)=(c_n,\dots,c_1)$ of $T$ is defined by $c_i=\area(T_i)$, where $T_i$ is the lattice path from the left most node in $T$ at height $i$ to the root, and $\area(T_i)$ is the number of boxes below this path in the smallest rectangle containing it. 
The coordinates of our canonical 2-dimensional associahedron are computed in~\Cref{fig_coordinates_canonical_2D}, and the polytope is illustrated on the right of~\Cref{fig_associahedron_Loday_canonical_2D}. 

\begin{figure}[htb]
    \centering
    \begin{tabular}{ccccc}
\nuTree{0,1,1,1}{0,1,1,1}{0.64}{}{-0.2/2/$0$/red/east,-0.2/1/$0$/blue/east}{}{} 
&
\nuTree{0,1,1,1}{0,0,2,1}{0.64}{}{-0.2/2/$0$/red/east,-0.2/1/$1$/blue/east}{}{
\filldraw[nearly transparent, blue] (1,1) -- (1,2) -- (0,2) -- (0,1);
}
&
\nuTree{0,1,1,1}{0,0,1,2}{0.64}{}{-0.2/2/$1$/red/east,-0.2/1/$2$/blue/east}{}{
\filldraw[nearly transparent, red] (1,2) -- (1,3) -- (0,3) -- (0,2);
\filldraw[nearly transparent, blue] (1,1) -- (1,3) -- (0,3) -- (0,1);
}
&
\nuTree{0,1,1,1}{0,0,0,3}{0.64}{}{-0.2/2/$2$/red/east,-0.2/1/$2$/blue/east}{}{
\filldraw[nearly transparent, red] (2,2) -- (2,3) -- (0,3) -- (0,2);
\filldraw[nearly transparent, blue] (1,1) -- (1,3) -- (0,3) -- (0,1);
}
&
\nuTree{0,1,1,1}{0,1,0,2}{0.64}{}{-0.2/2/$2$/red/east,-0.2/1/$0$/blue/east}{}{
\filldraw[nearly transparent, red] (2,2) -- (2,3) -- (0,3) -- (0,2);
}
\\
\rule{0pt}{3ex}    
(0,0) &
(0,1) &
(1,2) &
(2,2) &
(2,0) 
\end{tabular}
    \caption{Coordinates of the vertices of our canonical realization of the 2-dimensional associahedron.}
    \label{fig_coordinates_canonical_2D}
\end{figure}

\begin{figure}[htb]
    \centering
    \input{figures/associahedron_Loday_2D}
    \qquad
    \input{figures/associahedron_canonical_2D}
    \caption{The 2-dimensional associahedron: Loday's realization (left) and our canonical realization (right). They are related by the map $(\ell_1,\ell_2,\ell_3) \to (c_2,c_1)$ where $c_1=\ell_1-1$ and $c_2=\ell_1+\ell_2-(1+2)$.}
    \label{fig_associahedron_Loday_canonical_2D}
\end{figure}

\begin{theorem}\label{thm_Loday_to_Ceballos}
    Loday's associahedron and our canonical realization of the associahedron are related by the simple affine transformation
    \begin{align*}
      \varphi:\mathbb{R}^{n+1}&\to \mathbb{R}^n \\
      \varphi(\ell_1,\dots,\ell_{n+1})&=(c_n,\dots,c_1)
    \end{align*}
    defined by $c_i=(\ell_1+\dots+\ell_i) - (1+\dots+i)$.
\end{theorem}

\begin{proof}
    We need to show that $\varphi(L(T))=C(T)$ for every plane binary tree $T$. This follows if we prove the following two properties:
    \begin{enumerate}
        \item $\varphi(L(T_0))=C(T_0)$ where $T_0$ be the bottom element of the Tamari lattice.
        \item $\varphi(L(T'))-\varphi(L(T))=C(T')-C(T)$ if the tree $T'$ is obtained from a tree $T$ by applying a tree rotation. 
    \end{enumerate}

    For Property (1), note that $L(T_0)=(1,2,\dots,n+1)$ and $C(T_0)=(0,\dots,0)$. Applying the map $\varphi$ to the first, we obtain $\varphi(L(T_0))=C(T_0)$ as desired.

    For Property (2), consider two trees $T$ and $T'$ related by a tree rotation as illustrated in~\Cref{fig_proof_Loday_Ceballos}, where $a,b,c,d,e$ are the number of leaves of the tree in the denoted parts.
    Let $L(T)=(\ell_1,\dots,\ell_{n+1})$ and $L(T')=(\ell_1',\dots,\ell_{n+1}')$ be their corresponding coordinates in Loday's associahedron. 
    
    \begin{figure}[htb]
        \centering
        \input{figures/proof_Loday_Ceballos}
        \caption{The two trees related by rotation in the proof of~\Cref{thm_Loday_to_Ceballos}.}
        \label{fig_proof_Loday_Ceballos}
    \end{figure}

    Note that the entries of $L(T)$ and $L(T')$ differ only at two positions, corresponding to the two nodes~$p,q\in T$ and $p',q'\in T'$.
    More precisely, we have the following. 

    In the tree $T$, we have that 
    \begin{itemize}
        \item $q\in T$ is the $(a+b)$th internal node in in-order. Therefore, $\ell_{a+b}=bc$.
        \item $p\in T$ is the $(a+b+c)$th internal node in in-order. Therefore, $\ell_{a+b+c}=(b+c)d$.
    \end{itemize}

    In the tree $T'$, we have that 
    \begin{itemize}
        \item $p'\in T'$ is the $(a+b)$th internal node in in-order. Therefore, $\ell_{a+b}'=b(c+d)$.
        \item $q'\in T'$ is the $(a+b+c)$th internal node in in-order. Therefore, $\ell_{a+b+c}'=cd$.
    \end{itemize}

    As a consequence, we deduce that
    \begin{align}
        \ell_{a+b}'-\ell_{a+b} &= bd \\
        \ell_{a+b+c}'-\ell_{a+b+c} &=-bd.
    \end{align}

    Furthermore, $\ell_i'-\ell_i=0$ for all other indices.
    In other words, the vector $L(T')$ is obtained from the vector $L(T)$ by adding $bd$ to the $(a+b)$th entry and subtracting $bd$ to the $(a+b+c)$th entry. 

    If we denote by $\varphi(L(T))=(\widetilde c_n, \dots,\widetilde c_1)$ and by $\varphi(L(T'))=(\widetilde c_n', \dots,\widetilde c_1')$, 
    Then,
    \begin{equation}\label{eq_proof_Loday}
    \widetilde c_i' =
    \begin{cases}
        \widetilde c_i+bd, & \text{if } a+b \leq i \leq a+b+c-1 \\
        \widetilde c_i, & \text{otherwise.}
    \end{cases}                
    \end{equation}

    On the other hand, the canonical coordinates satisfy an analog formula.
    Let $C(T)=(c_n,\dots,c_1)$ and $C(T')=(c_n',\dots,c_1')$ be the coordinates of $T$ and $T'$ in our canonical realization.
    
    The entries of $C(T)$ and $C(T')$ differ exactly at the positions corresponding to heights of the left most nodes (internal nodes or leaves) of $T$ in each row of the triangle labeled $c$ in~\Cref{fig_proof_Loday_Ceballos}. These are exactly the heights from $i=a+b$ to $i=a+b+c-1$. For all these heights, the area of the path~$T_i$ increases by the red area illustrated in the figure, which is equal to $bd$. Therefore,
    \begin{equation}\label{eq_proof_Ceballos}
     c_i' =
    \begin{cases}
         c_i+bd, & \text{if } a+b \leq i \leq a+b+c-1 \\
         c_i, & \text{otherwise.}
    \end{cases}                
    \end{equation}

    Comparing Equations~\eqref{eq_proof_Loday} and~\eqref{eq_proof_Ceballos}, we deduce that $\widetilde c_i'-\widetilde c_i=c_i'-c_i$. And so,
    \begin{align}
        \varphi(L(T'))-\varphi(L(T))=C(T')-C(T)
    \end{align}
    as desired.  
\end{proof}

\section{Enumerative properties}
The purpose of this section is to present some enumerative properties about the $h$-vector of the alt~$\nu$-Tamari complex and the $f$-vector of the alt $\nu$-associahedron~$\altAsso{\nu}{\delta}$. Interestingly, for a fixed $\nu$, these two vectors are independent of the choice of $\delta$. We will see some relations to $\nu$-Schröder paths and the $\nu$-Narayana numbers. 

\subsection{The $h$-vector of the alt $\nu$-Tamari complex}
Let \(\Delta\) be a \(d\)-dimensional simplicial complex. The \(f\)-vector of \(\Delta\) is
\[
f(\Delta)=(f_{-1},f_0,\dots,f_{d}),
\]
where \(f_i\) denotes the number of \(i\)-dimensional faces of \(\Delta\), and \(f_{-1}=1\).

The \(h\)-vector of \(\Delta\) is
\[
h(\Delta)=(h_0,h_1,\dots,h_{d+1}),
\]
defined by the relation
\[
\sum_{i=0}^{d+1} h_i t^{d+1-i}
=
\sum_{i=0}^{d+1} f_{i-1}(t-1)^{d+1-i}.
\]
    \begin{theorem}
    \label{thm_hvector}
        For a fixed $\nu$, all alt $\nu$-Tamari complexes $\altTamComplex{\nu}{\delta}$ have the same~\mbox{$h$-vector}. 
        \begin{enumerate}
            \item $h_i$ is equal to the number of $\nu$-paths with exactly $i$ valleys.
            \item $h_i$ is equal to the number of elements of $\altTam{\nu}{\delta}$ with exactly~$i$ upper covers.
        \end{enumerate}
        Both quantities are independent of $\delta$.
        The numbers $h_0,h_1,\dots$ are called the \defn{$\nu$-Narayana numbers}.
    \end{theorem}

\begin{proof}
    Let $\delta,\delta'$ be two increment vectors for a fixed path $\nu$.  As shown in~\Cref{lem_changing_U}, the polytopes~$\subpolytope{U_{\delta,\nu}}$ and $\subpolytope{U_{\delta',\nu}}$ are integrally equivalent. Furthermore, $\altTamComplex{\nu}{\delta}$ and $\altTamComplex{\nu}{\delta'}$ give two unimodular triangulations of these two polytopes. Since all unimodular triangulations of a polytope have the same $h$-vector, and this property is preserved under integrally equivalence, then all alt $\nu$-Tamari complexes must have the same $h$-vector. Therefore, it is sufficient to prove the result for one particular case of $\delta$. Taking the extreme case $\delta_i=\nu_i$ recovers the $\nu$-Tamari complex, for which this result is known~\cite[Theorem~4.6]{ceballos_geometry_2019}.  
\end{proof} 

\begin{remark}
    The $\nu$-Narayana numbers in~\Cref{thm_hvector} were shown to be the entries of the $h$-vector of the $\nu$-Tamari complex in~\cite{ceballos_geometry_2019}. They proved this by showing that any linear extension of the $\nu$-Tamari lattice gives a shelling order of the facets of the $\nu$-Tamari complex. Using the same techniques one can show that this results also holds for the alt $\nu$-Tamari lattice.  
\end{remark}

The following proposition allows us to count the number of interior faces of the alt $\nu$-Tamari complex in terms of its $h$-vector. 
\begin{proposition}\label{prop_hpolynomial}
    The $h$-polynomial $h(t)= \sum_{i=0}^{d+1} h_i t^i$ of the all alt $\nu$-Tamari complex $\altTamComplex{\nu}{\delta}$ satisfies
    \begin{equation}\label{eq_h_evaluation_interior_faces}
        h(t+1) =
        \sum_{i=0}^{d} f_{d-i}^{int} t^{i},
    \end{equation}
    where $f_i^{int}$ is the number of $i$-dimensional interior faces of $\altTamComplex{\nu}{\delta}$. 
\end{proposition}

\begin{proof}
    By~\cite[Corollary~3.2]{CeballosMuhle2022}, we obtain 
    \begin{equation}
        t^{d+1}h\left( \frac{t+1}{t}\right) =
        \sum_{i=1}^{d+1} f_{i-1}^{int} t^i.
    \end{equation}
    Evaluating at $\frac{1}{t}$ and then multiplying by $t^{d+1}$ yields 
    \begin{equation}
        h(t+1) =
        \sum_{i=1}^{d+1} f_{i-1}^{int} t^{d+1-i}
        =
        \sum_{i=0}^{d} f_{i}^{int} t^{d-i}
        =
        \sum_{i=0}^{d} f_{d-i}^{int} t^{i}.
    \end{equation}
\end{proof}

\begin{corollary}\label{cor_interior_faces}
    The number $f_{d-i}^{int}$ of codimension $i$ interior faces of $\altTamComplex{\nu}{\delta}$ is equal to the number of pairs $(\mu,V)$ such that $\mu$ is a $\nu$-path and $V$ is a subset of valleys of $\mu$, with $|V|=i$. 
\end{corollary}

\begin{proof}
    By~\Cref{thm_hvector}, the $h$-polynomial of $\altTamComplex{\nu}{\delta}$ is
    \[
    h(t)=\sum_{\mu} t^{\operatorname{valleys}(\mu)}
    \]
    where the sum runs over all $\nu$-paths $\mu$ and $\operatorname{valleys}(\mu)$ is the number of valleys of $\mu$.
    Therefore, 
    \begin{align*}
    h(t+1) = \sum_{\mu} (t+1)^{\operatorname{valleys}(\mu)}     = \sum_{(\mu,V)} t^{|V|},
    \end{align*}
    where the sum runs over all pairs $(\mu,V)$ consisting of a $\nu$-path $\mu$ and a subset $V$ of valleys of $\mu$. 
    Combining this with~\Cref{prop_hpolynomial} yields the result.
\end{proof}

\subsection{The $f$-vector of the alt $\nu$-associahedron}
Let $\nu$ be a lattice path and $\delta$ be an increment vector with respect to $\nu$. The $f$-vector of the alt $\nu$-associahedron $\altAsso{\nu}{\delta}$ is the vector 
\[
f(\altAsso{\nu}{\delta})=
(f_{d}^{int},f_{d-1}^{int},f_{d-2}^{int},\dots),
\]
where $f_i(\altAsso{\nu}{\delta})=f_{d-i}^{int}$ is the number of codimension $i$ interior faces of $\altTamComplex{\nu}{\delta}$. 
In particular, the number of vertices is $f_{d}^{int}$, equal to the number of $\nu$-paths, which is independent of the choice of $\delta$. The following result generalizes this property.  

\begin{proposition}\label{prop_f_altnuAsso}
    For a fixed $\nu$, 
    \begin{enumerate}
        \item all alt $\nu$-associahedra $\altAsso{\nu}{\delta}$ have the same~\mbox{$f$-vector}. 
        \item $f_i(\altAsso{\nu}{\delta})=$ number of pairs $(\mu,V)$ such that $\mu$ is a $\nu$-path and $V$ is a subset of valleys of $\mu$, with $|V|=i$.
    \end{enumerate}
\end{proposition}

\begin{proof}
    By~\Cref{prop_hpolynomial}, the vector $(f_{d}^{int},f_{d-1}^{int},f_{d-2}^{int},\dots)$ is completely determined by the $h$-vector of $\altTamComplex{\nu}{\delta}$, which is independent of $\delta$ by~\Cref{thm_hvector}. This proves part (1). Since $f_i(\altAsso{\nu}{\delta})=f_{d-i}^{int}$,  Part (2) follows from~\Cref{cor_interior_faces}.
\end{proof}

This result leads to a nice  description of the faces and the $f$-vector of the alt $\nu$-associahedron $\altAsso{\nu}{\delta}$ in terms of $\nu$-Schröder paths. Our description generalizes the results in~\cite{vonbell_schroder_2021}.  
    
A \defn{$\nu$-Schröder path} is a lattice path using north steps~$(0,1)$, east steps~$(1,0)$, and diagonal steps~$(1,1)$, such that it stays weakly above $\nu$ and has the same initial and final points as $\nu$.  

\begin{theorem}\label{thm_schroderpaths}
    The $i$-dimensional faces of the alt $\nu$-associahedron $\altAsso{\nu}{\delta}$ are in bijection with the $\nu$-Schröder paths using $i$ diagonal steps.   
\end{theorem}

\begin{proof}
    By~\Cref{prop_f_altnuAsso}, the number of $i$-dimensional faces of $\altAsso{\nu}{\delta}$ is
    \[
    f_i(\altAsso{\nu}{\delta})= |\{
    (\mu,V):\, \text{$\mu$ is a $\nu$-path, $V$ is a subset of valleys of $\mu$}
    \}|
    \]
    Each such a pair $(\mu,V)$ can be thought as a $\nu$-path $\mu$ with some marked valleys $V$. Replacing each~$NE$ at a marked valley by a diagonal step gives a $\nu$-Schröder path. Viceversa, each $\nu$-Schröder path can be obtained this way by a unique pair $(\mu,V)$.  
\end{proof}

This result can be interpreted geometrically as follows. Each face of the alt $\nu$ associahedron is determined by
\begin{itemize}
    \item its minimal element: a $\nu$-path $\mu$, and
    \item the edges incident to the minimal element in that face: a subset $V$ of valleys of $\mu$, describing the rotation along those edges.
\end{itemize} 

So, we can label each face of $\altAsso{\nu}{\delta}$ by pairs $(\mu,V)$ such that $\mu$ is a $\nu$-path and $V$ is a subset of valleys of $\mu$. The dimension of the face corresponding to the pair $(\mu,V)$ is then equal to $|V|$.
This is illustrated for an example on the top of~\Cref{fig_interior_faces_schroederpaths}. 
The pairs $(\mu,V)$ can be though of as $\nu$-paths with some marked valleys. 
We obtain a labeling of the faces of $\altAsso{\nu}{\delta}$ by $\nu$-Schröder paths by replacing $NE$ at each marked valley by a diagonal step. This is illustrated on the bottom of~\Cref{fig_interior_faces_schroederpaths}.
The dimension of a face is then equal to the number of diagonal steps of the corresponding $\nu$-Schröder path.

\begin{figure}[htb]
    \begin{center}
        \input{figures/interior_faces_schroederpaths1}
        
        \vspace{3mm}
        \input{figures/interior_faces_schroederpaths2}
    \end{center}
    \caption{Faces of the alt $\nu$-associahedron $\altAsso{\nu}{\delta}$ labeled by $\nu$-Schröder paths, for~$\nu=ENEEN=(1,2,0)$ and $\delta=(1,0)$. 
    }
    \label{fig_interior_faces_schroederpaths}
\end{figure}

\section*{Acknowledgment}
I am very grateful to Matias von Bell, Cl{\'e}ment Chenevi{\`e}re, Sergio Fernandez de soto, Matthias Müller, Vincent Pilaud, and Yannic Vargas  for helpful conversations on related topics. 

\newpage
\bibliographystyle{plain}
\bibliography{biblio}

\newpage
\appendix

\section{Some 3D-examples}\label{sec_3D_examples}
In this appendix we show several examples of alt $\nu$-associahedra using our canonical coordinates. 

Figure~\ref{fig_some_3D_nuAssociahedra} shows four examples of $\nu$-associahedra, while Figure~\ref{fig_some_3D_alt_nuAssociahedra} shows all alt $\nu$-associahedra for $\nu=ENENEN=(1,1,1,0)$.
The examples of these two figures are shown in more detail in the remaining figures, where the vertices are labeled by points, coordinates, paths, and trees. 


\begin{figure}[h]
    \centering
    \input{figures/some_3D_nuAssociahedra}
    \caption{Some 3D $\nu$-associahedra.}
    \label{fig_some_3D_nuAssociahedra}
\end{figure}


\begin{figure}[h]
    \centering
    \input{figures/some_3D_alt_nuAssociahedra}
    \caption{All alt $\nu$-associahedra for $\nu=ENENEN=(1,1,1,0)$.}
    \label{fig_some_3D_alt_nuAssociahedra}
\end{figure}


\begin{figure}[h]
    \centering
    \input{figures/nu_associahedron_1110_all}
    \caption{The $ENENEN$-associahedron with vertices labeled by points, coordinates, paths, and trees.}
    \label{fig_nu_associahedron_1110_all}
\end{figure}


\begin{figure}[h]
    \centering
    \input{figures/nu_associahedron_1120_all}
    \caption{The $ENENEEN$-associahedron with vertices labeled by points, coordinates, paths, and trees.}
    \label{fig_nu_associahedron_1120_all}
\end{figure}


\begin{figure}[htb]
    \centering
    \begin{tabular}{c}
        \input{figures/nu_associahedron_1220}
         \\[10mm] 
        \input{figures/nu_associahedron_1220_coordinates}
    \end{tabular}
    \caption{The $ENEENEEN$-associahedron with vertices labeled by points and coordinates.}
    \label{fig:enter-label}
\end{figure}

\begin{figure}[htb]
    \centering
    \begin{tabular}{c}
        \input{figures/nu_associahedron_1220_paths}
         \\[10mm]
        \input{figures/nu_associahedron_1220_trees}        
    \end{tabular}
    \caption{The $ENEENEEN$-associahedron with vertices labeled by paths and trees.}
    \label{fig:enter-label}
\end{figure}


\begin{figure}[htb]
    \centering
    \begin{tabular}{c}
        \input{figures/nu_associahedron_2110}
         \\[10mm] 
        \input{figures/nu_associahedron_2110_coordinates}
    \end{tabular}
    \caption{The $EENENEN$-associahedron with vertices labeled by points and coordinates.}
    \label{fig:enter-label}
\end{figure}

\begin{figure}[htb]
    \centering
    \begin{tabular}{c}
        \input{figures/nu_associahedron_2110_paths}
         \\[10mm]
        \input{figures/nu_associahedron_2110_trees}        
    \end{tabular}
    \caption{The $EENENEN$-associahedron with vertices labeled by paths and trees.}
    \label{fig:enter-label}
\end{figure}


\begin{figure}[h]
    \centering
    \input{figures/alt_nu_asso_1110_000_all}
    \caption{The alt $\nu$-associahedron with vertices labeled by points, coordinates, paths, and trees, for $\nu=ENENEN=(1,1,1,0)$ and $\delta=(0,0,0)$.}
    \label{fig_alt_nu_asso_1110_000_all}
\end{figure}


\begin{figure}[h]
    \centering
    \input{figures/alt_nu_asso_1110_100_all}
    \caption{The alt $\nu$-associahedron with vertices labeled by points, coordinates, paths, and trees, for $\nu=ENENEN=(1,1,1,0)$ and $\delta=(1,0,0)$.}
    \label{fig_alt_nu_asso_1110_100_all}
\end{figure}


\begin{figure}[h]
    \centering
    \input{figures/alt_nu_asso_1110_010_all}
    \caption{The alt $\nu$-associahedron with vertices labeled by points, coordinates, paths, and trees, for $\nu=ENENEN=(1,1,1,0)$ and $\delta=(0,1,0)$.}
    \label{fig_alt_nu_asso_1110_010_all}
\end{figure}

\end{document}